\documentclass[11pt,bezier,amstex]{article}  

\newcommand{\pr}{\rightarrow}

\newcommand{\ba}{\begin{array}}
\newcommand{\ea}{\end{array}}

\newcommand{\varp}{\varphi}
\newcommand{\eps}{\varepsilon}

\newenvironment{inspring}[1]%
{\begin{list}{}{\setlength{\rightmargin}{0cm}
                \setlength{\listparindent}{0cm}
                \settowidth{\labelwidth}{\mbox{#1}}
                \setlength{\leftmargin}{1.1\labelwidth}
                \setlength{\labelsep}{.1\labelwidth}}}%
{\end{list}}
\newcommand{\ITEM}[1]{\item[#1\hfill]}
\newcommand{\bi}[1]{\begin{inspring}{#1}}
\newcommand{\ei}{\end{inspring}}

\newcommand{\dsum}{\displaystyle \sum}

\newcommand{\beq}{\begin{equation}}
\newcommand{\eq}{\end{equation}}
\catcode`\@=11

\font\tenmsa=msam10 \font\sevenmsa=msam7 \font\fivemsa=msam5
\font\tenmsb=msbm10 \font\sevenmsb=msbm7 \font\fivemsb=msbm5
\newfam\msafam
\newfam\msbfam
\textfont\msafam=\tenmsa  \scriptfont\msafam=\sevenmsa
  \scriptscriptfont\msafam=\fivemsa
\textfont\msbfam=\tenmsb  \scriptfont\msbfam=\sevenmsb
  \scriptscriptfont\msbfam=\fivemsb

\def\Bbb{\ifmmode\let\next\Bbb@\else
 \def\next{\errmessage{Use \string\Bbb\space only in math mode}}\fi\next}
\def\Bbb@#1{{\Bbb@@{#1}}}
\def\Bbb@@#1{\fam\msbfam#1}
\newcommand{\dR}{{\Bbb R}}

\usepackage{color}              
\usepackage{graphicx,setspace}
\usepackage[]{amsmath}
\usepackage{amssymb}
\usepackage{fullpage,psfrag}
\usepackage{hyperref}
\usepackage{enumerate, amsfonts, amsthm, bbm,fancybox}
\usepackage{fullpage}
\numberwithin{equation}{section}

\definecolor{MyDarkBlue}{rgb}{0,0.08,0.50}
\definecolor{BrickRed}{rgb}{0.65,0.08,0}

\hypersetup{
colorlinks=true,       
    linkcolor=MyDarkBlue,          
    citecolor=BrickRed,        
    filecolor=red,      
    urlcolor=cyan           
}

\numberwithin{equation}{section}

\renewcommand{\geq}{\geqslant}
\renewcommand{\leq}{\leqslant}
\newtheorem{thm}{Theorem}

\newtheorem{prop}[thm]{Proposition}
\newtheorem{lem}[thm]{Lemma}
\theoremstyle{definition}

\newcommand{\prob}[1]{{\mathbb{P}}\left( #1 \right)}

\newcommand{\eqan}[1]{\begin{align} #1 \end{align}}

{\begin{Sbox}\begin{minipage}}%
{\end{minipage}\end{Sbox}\fbox{\TheSbox}}

\newcommand{\rv}{\text}

\begin{document}

\title{{\Large {\bf Staffing many-server systems with admission control and retrials (report version)}}}

\author{
A.J.E.M. Janssen\footnotemark[1] \and
J.S.H. van Leeuwaarden\footnotemark[2]
        }

\date{\today}

\maketitle

\footnotetext[1]{Eindhoven University of Technology and Eurandom, Department of Mathematics and Computer Science and Department of Electrical Engineering, P.O. Box 513, 5600 MB Eindhoven, The Netherlands. E-mail a.j.e.m.janssen@tue.nl
}

\footnotetext[2]{Eindhoven University of Technology, Department of Mathematics and Computer Science, P.O. Box 513, 5600 MB Eindhoven, The Netherlands. E-mail j.s.h.v.leeuwaarden@tue.nl
}

\begin{abstract}
In many-server systems it is crucial to staff the right number of servers so that targeted service levels are met.
These staffing problems typically lead to constraint satisfaction problems that are hard to solve. During the last decade, a powerful many-server asymptotic theory has been developed to solve such problems and optimal staffing rules are known to obey the square-root staffing principle. This paper develops many-server asymptotics in the so-called QED regime, and presents refinements to many-server asymptotics and square-root staffing for a Markovian queueing model with admission control and retrials.

\vspace{1.5mm}

\noindent {\bf Keywords}: Many-server systems, QED regime, Halfin-Whitt regime, heavy traffic, diffusion limits, admission control, square-root staffing, optimality gap, asymptotic dimensioning

\vspace{1.5mm}

\noindent {\bf AMS 2000 Subject Classification}: 60K25, 68M10, 41A60.

\end{abstract}

\section{Introduction}\label{sec:asy}

A key challenge in large many-server systems is to balance the trade-off between
 operational costs and quality-of-service offered to customers. There is by now a vast literature on the asymptotic analysis of many-server systems. Many papers describe
limiting behavior of a sequence of queues, which is then used to approximate the
characteristics of a member of the sequence, i.e., the performance of a finite-sized queueing system. Depending on how this
sequence is parameterized, its limiting behavior is different, giving rise to different approximations (see \cite{borst}
and \cite{mandelbaum}). One of the most \rv{popular} approximations
arises in the Quality-and-Efficiency-Driven (QED) regime, also known as the Halfin-Whitt regime \cite{halfinwhitt}, in which
the number of servers $s$ and the offered workload $\lambda$ are related
according to a square-root principle $s=\lambda+\beta \sqrt{\lambda}$, for a
constant $\beta$, and $s$ (and $\lambda$) are taken to infinity.

In this paper we consider a Markovian many-server system with admission control in the QED regime. We consider an admission control policy that lets an arriving customer enter the system according to a probability depending on the queue length. In particular, a customer meeting upon arrival $k$ other customers is admitted with probability $p_k$, and we shall allow for a wide range of such admission policies characterized by $(p_k)_{k\geq s}$.
For such admission-controlled systems, we consider the QED regime and the square-root staffing rule $\lambda=s-\gamma \sqrt{s}$ for some fixed constant $\gamma$. This is asymptotically (as $s\to \infty$), and for all practical purposes, equivalent with setting $s=\lambda+\beta\sqrt{\lambda}$ for some fixed constant $\beta$. The reason for staffing via $\lambda$ instead of $s$ is that it allows for mathematically more elegant derivations. All results in this paper regarding $\lambda=s-\gamma \sqrt{s}$ can be converted into results for $s=\lambda+\beta \sqrt{\lambda}$  at the cost of increased notational burden.

Since certain customers are rejected upon arrival, it seems natural to extend the model with the feature that rejected customers can reattempt. The modeling of reattempts or retrials is known to be challenging \cite{Cohen:1957,FaT:1997}, which is why one often resorts to computational approaches \cite{Artalejo-GC:2008}. These numerical approaches face increasing numerical difficulties when the number of servers becomes large, which is precisely the regime we are interested in. Therefore, we combine the QED regime with a limiting regime for slow retrials, meaning that rejected customers reattempt after a relatively (compared to the time scale of the system) long time. The combination of these two asymptotic regimes leads to a tractable model with a closed-form solution, for which we are able to derive QED approximations for some of the relevant performance measures.

We leverage these QED approximations to obtain results for staffing problems. The QED regime is particulary suited for staffing many-server systems, as it combines large capacity with high utilization. This appealing feature was exploited in \cite{borst} that introduced a rigorous asymptotic framework for applying the
square-root staffing principle to two classes of problems: constraint
satisfaction and cost minimization. In \cite{borst} it was observed that square-root staffing based on QED approximations
is effective and accurate over a wide range of system parameters
for the Erlang C model. We apply the same approach to develop square-root staffing rules for several constraint satisfaction problems for the system at hand.\\

\noindent{\bf Square-root staffing.}
The core of staffing problems in many-server systems is to determine the right trade-off between quality and capacity.
Quality is formulated in terms of some targeted service level. Take as an example the delay probability $D_F(s,\lambda)$ (see \eqref{3.4}). A large delay probability is perceived as negative, and the targeted service level could be to keep the delay probability below some value $\eps$. The smaller $\eps$, the better the offered service. Once the targeted service level is set, the objective from the system's perspective is to determine
the largest load $\lambda$ (or the lowest staffing level $s$) such that the target $D_F(s,\lambda)\leq\eps$ is met. This is what we have referred to
as a constraint satisfaction problem.


The delay probability is a function of the two model parameters $s$ and $\lambda$, and of the admission policy.
Because the delay probability is a continuous and monotone increasing function in $\lambda$, the constraint satisfaction problem is equivalent to finding the
$\lambda_{{\rm opt}}$ such that $D_F(s,\lambda_{{\rm opt}})=\eps$. To solve this inverse problem, we shall invoke the theory of asymptotic dimensioning introduced in \cite{borst} and extended in the first part of this paper to admission control and retrials. This theory fully exploits the QED regime for large systems, in a way that reduces considerably the complexity of the inverse problem. That is, in the QED regime (with $\lambda=s-\gamma\sqrt{s}$, $s\rightarrow\infty$, $\gamma$ not scaling with $s$), the performance measures in our model can be approximated by their limiting counterparts. For instance,
$D_F(s,\lambda)$ can be approximated by some function $D_*(\gamma)$ that only depends on $\gamma$ (and no longer on $s$ or $\lambda$). Hence, the inverse problem can then be approximatively solved by searching for the $\gamma_*$ such that
$D_*(\gamma_*)=\eps$, and then setting the load according to $\lambda_*=s-\gamma_*\sqrt{s}$. We refer to this rule as \textit{conventional square-root staffing}. In this asymptotic approach, one expects that the better the approximation $D_F(s,\lambda)\approx D_*(\gamma)$, the smaller the error $|\lambda_{{\rm opt}}-\lambda_*|$.
Based on the QED regime, one \rv{also} expects the approximation $\lambda_*$ to be accurate for large values of $s$, i.e.,~for large-scale systems with many servers. Indeed, we prove that $|\lambda_{{\rm opt}}-\lambda_*|=O(1)$, where a function
$f(\lambda)=O(g(\lambda))$  if
$\limsup_{\lambda\rightarrow\infty}|f(\lambda)/g(\lambda)|<\infty$.\\


\noindent{\bf Refined staffing.}
We shall also derive refined staffing rules, for which we first develop refined QED approximations for the objective function, and then characterize the approximative solutions to the constraint satisfaction problems. The refined staffing rules are of the form
\begin{equation}
\label{refined}
\lambda_\bullet = s-\gamma_*\sqrt{s} + r_\bullet,
\end{equation}
with $r_\bullet$ a simple function of  $\gamma_*$, $s$ and  $\eps$. We shall uniquely identify $r_\bullet$ (for each considered constraint satisfaction problem), and prove that the refined staffing level in \eqref{refined} satisfies
\begin{equation}\label{optgap}
\lambda_{{\rm opt}}-\lambda_\bullet=O(s^{-1/2}),
\end{equation}
We refer to the order term that expresses the difference between the exact optimal staffing level and the approximate staffing level as the {\it optimality gap}. Hence, the optimality gap of $\lambda_\bullet$ is $O(s^{-1/2})$, which suggests that the staffing level $\lambda_\bullet$ \rv{becomes more accurate as $s$ increases}.
Note that $\lambda_\bullet=\lambda_*+\gamma_\bullet$. Since the optimality gap of the conventional staffing level $\lambda_*$ equals  $O(1)$, we can expect that $\lambda_\bullet$  should be a more accurate prescription than $\lambda_*$. In addition,
because $\gamma_\bullet$ in fact describes the optimality gap of
$\lambda_*$, or more precisely,
 \begin{equation}\label{optgap}
\lambda_{\rm opt}-\lambda_*=\gamma_\bullet+O(s^{-1/2}),
\end{equation}
it allows us to perform an analytical assessment of the accuracy of both
conventional and refined square-root staffing.\\

\noindent {\bf The influence of admission policies and retrials.}
Our main results are captured in Theorems \ref{thmasss} and \ref{thmasss2}, which formally establish the necessary results for both conventional staffing rules and their refinements. These two theorems establish that as the system size $s$ increases, the difference between the true maximal load that adheres to the constraint $\lambda_{\rm opt}$, and the conventional square-root staffing prescription $\lambda_*$, converges to the real number $\gamma_\bullet$, see \eqref{optgap}. We stress that these results hold across the large class of admission control policies considered in this paper, and for the systems with and without retrials.

We obtain the explicit expression of $\gamma_\bullet$
in each case, and show that it is always a simple function of $\gamma_*$. We can therefore investigate numerically the impact of the refinements in comparison with conventional square-root staffing. It turns out that, unlike in the classical Erlang C model \cite{borst}, the refinement $\gamma_\bullet$ is significant in many cases, due to different system parameters. Our findings further suggest that in the presence of admission control, more care needs to be taken in applying many-server asymptotic results to small or moderate size systems. In particular, when the admission control becomes more lenient, and hence on average more customers are admitted, our numerical findings suggest that the refinements to the conventional square-root staffing are necessary.\\

\noindent{\bf Overview of the paper.}
In this paper we apply the QED regime to a Markovian many-server system with admission control and retrials. We analyze this system in the limiting regime, in a similar spirit as was done for the Erlang C model  \cite{halfinwhitt,companion}, the Erlang B model \cite{erlangb} and the Erlang A model (with abandonments) \cite{bo}. In the recent paper \cite{ajl} on a many-server system with retrials, the admission control policy that rejects all delayed customers (loss system) was  analyzed. The current paper strongly builds on several results obtained in \cite {ajl}.
Compared to  earlier studies \cite{halfinwhitt,companion,erlangb,bo,ajl}, the system in this paper brings about additional mathematical challenges, because of the effects of rejection and reattempts. In short, we make the following contributions:
\begin{itemize}
\item[(i)] We consider two stationary performance measures: The probability that an arriving customer finds all servers occupied $D_F(s,\lambda)$, and the probability that an arriving customers is rejected $D_F^R(s,\lambda)$.
For both performance measures we derive in Theorems \ref{thm:qed} and \ref{thm:qedretrials}  the limiting expressions in the QED regime, for the cases with and without retrials.
    \item[(ii)]
We next consider dimensioning problems of the type:  For fixed $s$, find $\lambda$ for which the rejection probability $D_F^R(s,\lambda)$ has some prescribed value. We solve this inverse problem both for finite $s$ and for $s\to \infty$ in the QED regime. The inverse problem in the QED regime is easier to solve, and provides a leading-order estimate, in terms of the hazard rate of the standard normal distribution, for the finite-$s$ inverse problem. We characterize explicitly the error made by replacing the finite-$s$ inverse problem by its QED counterpart, by deriving correction terms in Theorems \ref{thmasss} and \ref{thmasss2}. These correction terms are then also used to develop the refined staffing rules and establish the optimality gaps.
\item[(iii)]
We then consider dimensioning problems of the type:  For fixed $s$, find $\lambda$ for which the carried traffic $\lambda(1-D_F^R(s,\lambda))$ has a prescribed value. These dimensioning problems are shown to have typically two solutions: A small $\lambda$ leads to a large probability of being admitted $1-D_F^R(s,\lambda)$, while a large $\lambda$ leads to a small admittance probability; see Theorems \ref{thm12} and \ref{thmbi}. This phenomenon of two solutions has also been observed in certain loss networks with alternative two-link routing \cite{bistability}.
\item[(iv)] The more technical part of this paper deals with incorporating the effect of retrials, for which an essential role is played by the generalized Cohen equation
$\Omega=(\lambda+\Omega)D_F^R(s,\lambda+\Omega)$ and its solution $\Omega$ representing the average rate of retrials. In the next section we explain in more detail the fundamental role of this equation in our study.
\end{itemize}

The paper is structured as follows. In Section \ref{sec:model} we describe the model. In Section \ref{sec:QED} we present the limiting expressions for the performance measures in the QED regime, and we use these results to deal with the dimensioning problems in Section \ref{sec:dim}.
Finally, we present in Section \ref{sec:stab} results for dimensioning problems with multiple solutions. The proofs of the main results are given in Sections  \ref{sec4}-\ref{proofthm}, while the proofs of supporting results are deferred to the appendix.

\section{Model description}\label{sec:model}
The basic model described in this section is taken from \cite{fa}.
Consider a system with $s$ parallel servers to which customers arrive according to a Poisson process with rate $\lambda$. The service times of customers are exponentially distributed with mean one. An {\it admission policy} dictates whether a customer is admitted to the system or rejected. A customer that finds upon arrival a free server is immediately assigned to that server, and leaves the system after service. A customer that finds upon arrival $k$ other customers in the system, $k\geq s$, is allowed to join the queue with probability $p_k$ and is rejected with probability $1-p_k$. In this way, the sequence $(p_k)_{k\geq s}$ defines the admission policy. Since we are interested in large many-server systems working at critical load, and hence serving many customers, the parameter $p_k$ should be interpreted as the fraction of customers admitted, instead of the probability that determines the fate of just one single customer. For the results presented in this paper we impose only mild conditions on the sequence $(p_k)_{k\geq s}$, allowing for a wide range of admission policies to be considered.

Under the above Markovian assumptions, and assuming that all interarrival times and service times are mutually independent, the system can be described as a birth-and-death process $(C(t))_{t \geq 0}$ with $C(t)$ the number of customers in the system at time $t$. The birth and death rates from state $k$ are $\lambda p_k$ (with $p_k=1$ for $k=0,\ldots,s-1$) and $\min(k,s)$, respectively.  Assuming the stationary distribution exists, with $\pi_k=\lim_{t\rightarrow\infty}\prob{C(t)=k}$, it readily follows from solving the detailed balance equations that
\eqan{
\pi_k=\left\{
        \begin{array}{ll}
          \pi_0 \frac{\lambda^k}{k!}, & \hbox{$1\leq k\leq s$,} \\
          \pi_0 \frac{\lambda^k}{s!s^{k-s}}\prod_{j=s}^{k-1} p_j, & \hbox{$k\geq s+1$,}
        \end{array}
      \right.
}
with
\eqan{\label{22222}
\pi_0^{-1}=\sum_{k=0}^s\frac{\lambda^k}{k!}+\frac{\lambda^s}{s!}F\Big(\frac{\lambda}{s}\Big)
}
and
\beq \label{3.3}
F(  x)=\sum_{n=0}^{\infty}p_s\cdots p_{s+n}x^{n+1}.
\eq
From \eqref{22222} it can be seen that the stationary distribution exists when $F(\frac{\lambda}{s})<\infty$. Since $ p_{s+n}\in[0,1]$, we have that $F(\frac{\lambda}{s})<\infty$ when $0\leq\lambda<s$.
 When $\lambda\geq s$ we need to be more careful.
The radius of convergence of the power series $F(x)$ is given by $1/P$ with
\beq \label{4.1}
P:= \ba[t]{c}{{\rm lim}~{\rm sup}}\\[-1mm]{\mbox{\footnotesize $n\pr\infty$}}\ea (p_s\cdot...\cdot p_{s+n})^{\frac{1}{n+1}}\in[0,1].
\eq
In a major portion of the main text, we assume the following condition:
\beq\label{condA}
P\in[0,1),\ F(\frac1P - 0)=\lim_{x\uparrow\frac1P}F(x)=\infty.
\eq
Under this condition  it can be easily observed from \eqref{3.3} that the stability condition for our system becomes
\beq \label{stab1}
\lambda \in [0,\lambda_P) \quad {\rm with} \quad \lambda_P=\frac{s}{P},
\eq
where $\lambda_P=\infty$ when $P=0$. The condition \eqref{condA} is certainly not as general as possible to develop the theory, but it excludes cases that need separate consideration, thereby distracting attention from the bottom line of the exposition. Stability is guaranteed when $\lambda<s$. Also, when
 $\lim_{k\to \infty}p_k=0$, we have $P=0$ and thus stability for all $\lambda\geq 0$.
Condition \eqref{condA} is also satisfied for the case $p_k=p\in(0,1)$ for all $k\geq s$, where $F(x)=px/(1-px)$ so that $F(\frac1P -0)=\infty$, with $P=p$, indeed. We exclude at this point the case $P=1$, which would for example occur in the case $p_k=1$, $k\geq s$, and the cases that $F(\frac1P -0)<\infty$. However, in Section \ref{sec4} and Appendix \ref{aA}, the results are proved under general conditions.

Let
  \begin{equation}\label{}
B(s,\lambda) =  \frac{\lambda^s/s!}{\sum_{k=0}^{s}\lambda^k/k!}
\end{equation}
denote the Erlang B formula representing the stationary blocking probability in an $M/M/s/s$ system.
A crucial performance measure is the stationary probability $D_F(s,\lambda)=\sum_{k=s}^\infty \pi_k $ that an arriving customer finds all servers occupied, given by
\beq \label{3.4}
D_F^{-1}(s,\lambda)=\frac{B^{-1}(s,\lambda)+F(\frac{\lambda}{s})}{1+F(\frac{\lambda}{s})},
\eq
where $D_F^{-1}$ is short-hand notation for $(D_F)^{-1}$.
Note that for $p_k=0$, $k\geq s$, the term $F(\lambda/s)$ vanishes, and  the probability $D_F(s,\lambda)$ reduces to $B(s,\lambda)$.  Also, for $p_k=1$, $k\geq s$, the probability $D_F(s,\lambda)$
reduces to the Erlang C formula given by
  \begin{equation}\label{}
C(s,\lambda)\ = \frac{\frac{\lambda^s}{(s-1)!(s-\lambda)}}{\sum_{k=0}^{s-1}\lambda^k/k!+\frac{\lambda^s}{(s-1)!(s-\lambda)}},
\end{equation}
representing the stationary delay probability in an $M/M/s$ system.

Another important performance measure is the stationary probability $D_F^R(s,\lambda)=\sum_{k=s}^\infty \pi_k p_k$ of being rejected, given by
\beq \label{3.5}
D_F^{-R}(s,\lambda)=\frac{B^{-1}(s,\lambda)+F(\frac{\lambda}{s})}{1+(1-s/\lambda)F(\frac{\lambda}{s})}.
\eq
It follows from results in Section \ref{sec4} that
\beq
\max\{0,1-\frac{s}{\lambda}\}\leq D_F^R(s,\lambda)\leq B(s,\lambda)\leq D_F(s,\lambda)\leq 1
\eq
for $0\leq\lambda< \lambda_P$. Moreover, as a consequence of condition \eqref{condA} we have
\beq
D_F(s,\lambda_P-0)=1, \quad D_F^R(s,\lambda_P-0)=1-P.
\eq

The birth-and-death process $(C(t))_{t \geq 0}$ relies on the assumption that rejected customers are considered lost. Alternatively, we could assume that rejected customers reattempt to enter the system after some time. In that case, rejected customers start producing reattempts until they are allowed to enter. Assume that periods between successive reattempts of    a  rejected customer  are exponentially distributed with rate $\mu$, independent of interarrival and service times.
The system can then be described as a two-dimensional
process $(C(t),N(t))_{t \geq 0}$ with $C(t)$ the number of customers in the system and $N(t)$ the number of rejected customers at time $t$. Under the above
assumptions this process is a continuous-time Markov chain
on the lattice infinite strip $\{0,1,\ldots,s\}\times \mathbb{Z}_+$.

Since the transition rates of this process clearly depend on the second coordinate, the process
$\{(C(t),N(t)) ; t \geq 0\}$ is hard to analyze.
In fact, even deriving the stationary distribution poses analytical difficulties, and no closed-form solution seems to be available. We therefore make the following assumption: Reattempts arrive to the system according to a Poisson process with rate $\Omega$, independent of the Poisson process of customers that arrive to the system for the first time. This assumption is also known as the {\it retrials see time averages} (RTA) approximation.
Under this assumption, the total flow of customers arriving to the system is a Poisson process with rate $\lambda+\Omega$. The unknown rate $\Omega$ should then be the solution to
\eqan{\label{genC}
\Omega=(\lambda+\Omega)D_F^R(s,\lambda+\Omega).
}
Equation \eqref{genC} is intuitively clear as it equates two expressions for the rate of reattempts. In some cases the RTA approximation can be theoretically justified. Cohen \cite{Cohen:1957} showed that the system with $p_k=0$ for $k\geq s$, in the limit as $\mu\downarrow 0$, behaves as an Erlang loss system, except with an increased arrival intensity. More specifically, as $\mu\downarrow 0$, the distribution
of the number of busy servers converges to the corresponding distribution
for the standard Erlang loss system $M/M/s/s$ (which is a truncated
Poisson distribution), but with increased arrival rate $\lambda+\Omega$, where $\Omega$ is defined as the solution to \eqref{genC}.  Indeed, in the case of infinitely long retrial times, it fits intuition that the flow of reattempts is independent from the flow of primary customers. For retrial queues with finite retrial times, the RTA approximation has proved useful and accurate for many retrial systems.
We shall refer to \eqref{genC} as the {\it generalized Cohen equation}.


\begin{thm}[Unique solution Cohen's equation]\label{thmm1}Under condition \eqref{condA},
there is a unique solution $\Omega_{s,F}(\gamma)$ of equation {\rm(\ref{genC})} for any $\lambda\in(0,s)$.
\end{thm}

A proof of Theorem \ref{thmm1} can be distilled from \cite[Section 3]{fa}. In Appendix~\ref{aA} we shall present a self-contained proof of Theorem \ref{thmm1} under general conditions.

%

 \section{QED limits}\label{sec:QED}
The QED regime for many-server systems refers to  scaling of the arrival rate $\lambda$ and the number of servers $s$ such that, while both $\lambda$ and $s$ increase toward infinity, the traffic intensity $\rho=\lambda/s$ approaches unity and
\begin{equation}\label{scaling}
(1-\rho)\sqrt{s}\rightarrow \gamma,
\end{equation}
where $\gamma$ is a fixed constant.
The scaling combines large capacity with high utilization. For the Erlang loss and delay systems, this kind of scaling leads to the  classical results
(see e.g.~\cite[Section 5.2]{erlangb}), for $\gamma\in(-\infty,\sqrt{s})$ fixed,
\beq \label{3.6}
\frac{1}{\sqrt{s}}\,B^{-1}(s,s-\gamma\sqrt{s})=\frac{\Phi(\gamma)}{\varp(\gamma)}+O\Bigl(\frac{1}{\sqrt{s}}\Bigr),
~~~~~~s\pr\infty,
\eq
and for $\gamma\in(0,\sqrt{s})$ fixed,
\begin{equation}\label{limresult}
\lim_{s\rightarrow\infty}C(s,s-\gamma\sqrt{s}) =\left(1+\gamma \frac{\Phi(\gamma)}{\phi(\gamma)}\right)^{-1},
\end{equation}
where  $\Phi(x)$ and $\varp(x)$ denote the standard normal cumulative distribution function  and density, respectively.

The following result will prove useful in establishing QED limiting results.
\begin{lem}[Two decompositions]\label{lemdec} For $\lambda \in [0,\lambda_P)$,
\eqan{
D_F^{-1}(s,\lambda)&=(1-q_\lambda ){B^{-1}(s,\lambda)}+{q_\lambda }{C^{-1}(s,\lambda)},
\label{4.29}\\
D_F^{-R}(s,\lambda)&=B^{-1}(s,\lambda)+\frac{q_\lambda }{1-q_\lambda }\,C^{-1}(s,\lambda),\label{4.30}
}
where
\beq \label{4.31}
q_\lambda =\frac{\frac{s}{\lambda}F(\frac{\lambda}{s})}{1+F(\frac{\lambda}{s})}.
\eq
\end{lem}

The proof of Lemma \ref{lemdec} is given in Appendix~\ref{aA}. It can also be obtained from more general results in \cite{nes}.
Note that the function $C(s,\lambda)$ is also defined for $\lambda>s$, while in the $M/M/s$ queue the stability condition is $\lambda<s$.
It is further shown in Appendix~\ref{aA} that $0\leq q_\lambda \leq1$, with $q_\lambda =1$ if and only if $p_k=1$, $k\geq s$. Thus, for instance, $D_F^{-R}$ always exceeds $B^{-1}$ and the excess is given by the second term of the right-hand side of (\ref{4.30}), which is the product of a factor entirely determined by the admission policy and the Erlang C formula. Also, $D_F^{-1}$ is a convex combination, with a $\gamma$-dependent convexity parameter $1-q_\lambda $, of the Erlang B and C formulae. When an admission policy is mild, implying that $q_\lambda $ is close to 1, we have that the Erlang C formula is dominant. When an admission policy is strict, the Erlang B formula is dominant. Aside from these general comments, the variety of weight functions $q_\lambda $ that can occur, see (\ref{4.31}), is rather substantial. In the QED regime, though, the Erlang B formula is always dominant. Indeed, in the QED limit, we have $\lambda/s\to 1$, and so $q_\lambda\to F(1)/(1+F(1))$ which is a finite number $\in(0,1)$ by our condition \eqref{condA}. Now \eqref{3.6} and \eqref{limresult}  show that $B^{-1}$ grows like $\sqrt{s}$ while $C^{-1}$ remains bounded as $s\to\infty$.

We now apply the scaling     \eqref{scaling} to the system with admission policy. We henceforth keep working with the notation for the QED regime in \eqref{scaling}, which is why we reformulate the stability condition \eqref{stab1} as
\beq \label{stab2}
\gamma\in (\gamma_P,\sqrt{s}]
\eq
with
\beq \label{4.2}
\gamma_P={-}\,\frac{1-P}{P}\,\sqrt{s}\in({-}\infty,0).
\eq


 \begin{thm}[QED limits without retrials]\label{thm:qed}
Under condition \eqref{condA},
 for $\lambda=s-\gamma\sqrt{s}$, with $\gamma\in (-\infty,\sqrt{s}]$ fixed, 
 \begin{align}
 \lim_{s\to \infty}\sqrt{s}D_F(s,s-\gamma\sqrt{s})&=(1+F(1))\frac{\phi(\gamma)}{\Phi(\gamma)}, \label{eq1}\\
  \lim_{s\to \infty}\sqrt{s}D_F^R(s,s-\gamma\sqrt{s})&=\frac{\phi(\gamma)}{\Phi(\gamma)}\label{eq2}.
 \end{align}
 \end{thm}
\begin{proof}
We have by continuity of $F(x)$ at $x=1$ that for fixed $\gamma\in (-\infty,\sqrt{s}]$
\beq \label{3.8}
F\Bigl(\frac{s-\gamma\sqrt{s}}{s}\Bigr)=F\Bigl(1-\frac{\gamma}{\sqrt{s}}\Bigr)=F(1)+o(1),~~~~~~ s\to\infty.
\eq
Therefore,
\beq \label{3.9}
\frac{1}{\sqrt{s}}\,D_F^{-1}(s,s-\gamma\sqrt{s})=\frac{1}{1+F(1)}~\frac{\Phi(\gamma)}{\varp(\gamma)}+ o(1),~~~~~~s\to\infty,
\eq
and
\beq \label{3.10}
\frac{1}{\sqrt{s}}\,D_F^{-R}(s,s-\gamma\sqrt{s})=\frac{\Phi(\gamma)}{\varp(\gamma)}+O\Bigl(\frac{1}{\sqrt{s}}\Bigr), ~~~~~~s\to\infty.
\eq
 This implies \eqref{eq1} and \eqref{eq2}.
\end{proof}

 A first observation is that the limiting expressions \eqref{eq1} and \eqref{eq2} are similar as for the Erlang B formula in \eqref{3.6},
 and the only difference between the limits of  $D_F$ and $B$ is the factor $1+F(1)$, which incorporates all information about the admission policy.

Let $\bar D_F(s,\lambda)=D_F(s,\lambda+\Omega)$ and $\bar D_F^R(s,\lambda)=D_F^R(s,\lambda+\Omega)$ with $\Omega$ as in \eqref{genC}. Hence,
$\bar D_F(s,\lambda)$ and $\bar D_F^R(s,\lambda)$ are the stationary probability that an arriving customer finds all servers occupied, and the stationary probability that an arriving customers is rejected, respectively, in the system {\it with} retrials using the RTA approximation.
\begin{thm}[QED limits with retrials]\label{thm:qedretrials}
Under condition \eqref{condA}, for $\lambda=s-\gamma\sqrt{s}$, with $\gamma\in (-\infty,\sqrt{s}]$  fixed, and with $\Omega$  defined as in \eqref{genC},
\eqan{
\lim_{s\rightarrow\infty}\frac{\Omega}{\sqrt{s}}=a
}
with $a$ the unique positive solution of the equation
\eqan{
a=\frac{\varp(\gamma-a)}{\Phi(\gamma-a)}.
}
Furthermore,
\begin{align}
 \lim_{s\to \infty}\sqrt{s}\bar D_F(s,s-\gamma\sqrt{s})&=(1+F(1))\frac{\phi(\gamma-a)}{\Phi(\gamma-a)}, \label{eq1x}\\
  \lim_{s\to \infty}\sqrt{s}\bar D_F^R(s,s-\gamma\sqrt{s})&=\frac{\phi(\gamma-a)}{\Phi(\gamma-a)}.\label{eq2x}
\end{align}
\end{thm}

The proof of Theorem \ref{thm:qedretrials} is presented in Section \ref{sec6}. Theorem \ref{thm:qedretrials} shows that the additional load due to retrials $\Omega$, for a system with many servers, is of the order $\sqrt{s}$. In particular, as the number of servers grows large, $\Omega$ is well approximated by $a\sqrt{s}$, where $a$ is a constant that no longer depends on $s$. This also means that for the overall retrial system the arrival rate $\lambda+\Omega$ is approximately $s-(\gamma-a)\sqrt{s}$.

\section{Dimensioning problems}\label{sec:dim}
First consider the situation without retrials, and the problem of finding the arrival rate $\lambda$ such that  the probability $D_F(s,\lambda)$ to find all servers occupied or the probability $D_F^R(s,\lambda)$ that service is denied altogether has a prescribed value. Here the number of servers and the admission policy, embodied by $F$, are assumed to be given.

\noindent{\bf Problem 1} For fixed $s,\eps$, find $\gamma$ such that
\beq \label{eqi}
\sqrt{s}\,D(s,s-\gamma\sqrt{s})=\eps \quad  {\rm with} \quad D=D_F \ {\rm or} \ D_F^R.
\eq

In Section \ref{sec4} it will be shown that, under condition \eqref{condA}, $D_F(s,s-\gamma\sqrt{s})$ decreases strictly from $1$ at $\gamma=\gamma_P$ to $0$ at $\gamma=\sqrt{s}$. Unfortunately, such a result does not hold for $D_F^R$: There are policies $F$ satisfying condition \eqref{condA} such that
$D_F^R(s,s-\gamma\sqrt{s})$ is not monotonic as a function of $\gamma$, see Appendix \ref{aA}. These policies are, however, rather rare. Relevant policies for which
$D_F^R(s,s-\gamma\sqrt{s})$ is monotonic include
\beq
p_k=p\in(0,1), \ k\geq s,
\eq
and, for some $N\geq s$,
\beq\label{4.3a}
p_k=\left\{
      \begin{array}{ll}
        1, & \hbox{$s\leq k\leq N$,} \\
        0, & \hbox{$k\geq N+1$.}
      \end{array}
    \right.
\eq
When $D_F^R(s,s-\gamma\sqrt{s})$ is monotonic, it decreases from $1-P$ at $\gamma=\gamma_P$ to $0$ at $\gamma=\sqrt{s}$.

We have the following result:
\begin{thm}[Unique solutions] \label{thm7} Under condition \eqref{condA},\\
{\rm (i)} Equation \eqref{eqi} with $D=D_F$ has a unique solution $\gamma=\gamma_{s,F}(\eps)$ for any $\eps\in(0,\sqrt{s})$.\\
Assuming further that $D_F^R(s,s-\gamma\sqrt{s})$ is monotonic in $\gamma$,\\
{\rm (ii)}  Equation \eqref{eqi} with $D=D_F^R$ has a unique solution $\gamma=\gamma_{s,F}^R(\eps)$ for any $\eps\in(0,(1-P)\sqrt{s})$.
\end{thm}

We next consider Problem 1 in the QED regime, and first introduce some definitions. Let for $\gamma\in(\gamma_P,\sqrt{s})$
\beq \label{4444}
g_{s,F}(\gamma):=\sqrt{s}D_F(s,s-\gamma\sqrt{s}), \quad g_{s,F}^R(\gamma):=\sqrt{s}D_F^R(s,s-\gamma\sqrt{s}).
\eq
Furthermore, define for $\gamma\in\mathbb{R}$,
\beq
g_\infty(\gamma)=\frac{\phi(\gamma)}{\Phi(\gamma)},
\eq
and define
 $\gamma_{\infty,F}(\eps)$ and $\gamma_\infty^R(\eps)$   as the solutions of
\beq \label{defdd}
(1+F(1))g_\infty(\gamma)={\eps}\quad {\rm and} \quad g_\infty(\gamma)={\eps},
 \eq
 respectively. It is well-known, see \cite[Subsection~4.1]{ajl} that $g_\infty(\gamma)$ strictly decreases from $+\infty$ at $\gamma=-\infty$ to $0$ at $\gamma=\infty$, and so both equations in \eqref{defdd} have unique solutions when $\eps>0$.

 In Theorem  \ref{thmasss} below, we give a limit  result
 for $\gamma_{s,F}(\eps)$ and $\gamma_{s,F}^R(\eps)$  as $s\to\infty$ that involves $\gamma_{\infty,F}(\eps)$ and $\gamma_{\infty,F}^R(\eps)$, respectively. For this result, the following observations are made.
 From \eqref{3.4} and \eqref{3.5} we get
\beq
g_{s,F}(\gamma)=\sqrt{s}B(s,s-\gamma\sqrt{s})\frac{1+F(1-\gamma/\sqrt{s})}{1+B(s,s-\gamma\sqrt{s})F(1-\gamma/\sqrt{s})},
\eq
and
 \beq
g_{s,F}^R(\gamma)=\sqrt{s}B(s,s-\gamma\sqrt{s})\frac{1-\frac{\gamma}{\sqrt{s}}(1-\gamma/\sqrt{s})^{-1}F(1-\gamma/\sqrt{s})}{1+B(s,s-\gamma\sqrt{s})F(1-\gamma/\sqrt{s})}.
\eq
Now we have from \cite[Theorem 14]{jagerman},
\beq\label{4999}
\sqrt{s}B(s,s-\gamma\sqrt{s})=g_\infty(\gamma)+\frac{1}{\sqrt{s}}h_\infty(\gamma)+O(s^{-1}),
\eq
where
\beq
h_\infty(\gamma)=-\tfrac13 \Big(\gamma^3+(\gamma^2+2)g_\infty(\gamma)\Big)g_\infty(\gamma),
\eq
and the $O$ in \eqref{4999} holds uniformly in any bounded set of $\gamma$'s. Using \eqref{4999}, together with
\beq
F(1-\tfrac{\gamma}{\sqrt{s}})=F(1)-\frac{\gamma}{\sqrt{s}}F'(1)+O(s^{-1}),
\eq
which holds because of assumption \eqref{condA}, the following result is established upon computation.

\begin{lem}Under condition \eqref{condA},
\begin{align}
g_{s,F}(\gamma)&=(1+F(1))g_\infty(\gamma)+\frac{1}{\sqrt{s}}h_{\infty,F}(\gamma)+O(s^{-1}), \label{4999a}\\
g_{s,F}^R(\gamma)&=g_\infty(\gamma)+\frac{1}{\sqrt{s}}h_{\infty,F}^R(\gamma)+O(s^{-1}),\label{4999b}
\end{align}
where
\begin{align}
h_{\infty,F}(\gamma)&=(1+F(1))h_\infty(\gamma)-(\gamma F'(1)+(1+F(1))F(1))g_\infty(\gamma),\\
h_{\infty,F}^R(\gamma)&=h_\infty(\gamma)-(\gamma +g_\infty(\gamma))g_\infty(\gamma)F(1).
\end{align}
The  $O$ in \eqref{4999a} and \eqref{4999b} holds uniformly in any bounded set of $\gamma$'s.
\end{lem}

\begin{thm}[Asymptotic dimensioning without retrials]\label{thmasss}Under condition \eqref{condA},
\begin{align}
\gamma_{s,F}(\eps)&=\gamma_{\infty,F}(\eps)+ \frac{1}{\sqrt{s}}\eta_{\infty,F}(\eps)+O(s^{-1}),\label{416}\\
\gamma_{s,F}^R(\eps)&=\gamma_\infty^R(\eps)+\frac{1}{\sqrt{s}}\eta_{\infty,F}^R(\eps)+O(s^{-1})\label{417},
\end{align}
with
\begin{align}
\eta_{\infty,F}(\eps)&=-\frac{h_{\infty,F}(\gamma_{\infty,F}(\eps))}{(1+F(1))g'_{\infty}(\gamma_{\infty,F}(\eps))},\label{418}\\
\eta_{\infty,F}^R(\eps)&=-\frac{h_{\infty,F}^R(\gamma_{\infty}^R(\eps))}{g'_{\infty}(\gamma_{\infty}^R(\eps))}. \label{419}
\end{align}
\end{thm}
Theorem \ref{thmasss} gives the limits of  $\gamma_{s,F}(\eps)$ and $\gamma_{s,F}^R(\eps)$ as $s\to\infty$, together with first-order corrections $\eta_{\infty,F}(\eps)$ and $\eta_{\infty,F}^R(\eps)$ that are simple functions of the limits $\gamma_{\infty,F}(\eps)$ and $\gamma_\infty^R(\eps)$. It is here instrumental to note that
\beq
g_\infty'(\gamma)=-g_\infty(\gamma)(\gamma+g_\infty(\gamma)).
\eq
The proof of Theorem \ref{thmasss} will be given in Section \ref{proofthm}.

We next discuss the numerical experiments that we conducted to illustrate the analytical results.
Remember that the objective of Problem 1 is to determine maximal sustainable load $\lambda$  such that
the rejection probability $D_F^R(s,\lambda)$ is below a threshold $\eps_s:=\eps/\sqrt{s}$.
In Section \ref{sec:asy} we have denoted the true maximal load by $\lambda_{\rm opt}$, and we have explained the concepts of square-root staffing and asymptotic dimensioning, in order to obtain accurate estimates of $\lambda_{\rm opt}$ that are asymptotically sharp in the QED regime.

The conventional
square-root staffing rule is to
use the QED approximation $\sqrt{s}D_F^R(s,\lambda)\approx D_*(\gamma)$,
obtain the solution to $D_*(\gamma)=\eps_s$, say $\gamma_*$, and then prescribe the
load as $\lambda_*=s-\gamma_*\sqrt{s}$. Theorem \ref{thmasss} allows for refined-square root staffing based on a better
QED approximation $\sqrt{s}D_F^R(s,\lambda)\approx D_\bullet(\gamma)$ and the solution $\gamma_\bullet$ to  $D_\bullet(\gamma)=\eps_s$.

For the asymptotic dimensioning sketched above and in Section \ref{sec:asy}, applied to Problem 1, we identify the following key functions and  parameters:
\begin{align}
D_*(\gamma) &=\frac{\phi(\gamma)}{\Phi(\gamma)},\\
D_\bullet(\gamma) &=D_*(\gamma)+\frac{1}{\sqrt{s}}h_{\infty,F}^R(\gamma),\\
\gamma_*&=\gamma_{\infty}^R(\eps),\\
\gamma_\bullet&=\gamma_*+ \frac{1}{\sqrt{s}}\eta_{\infty,F}^R(\eps),
\end{align}
and using the square-root rule $\lambda=s-\gamma\sqrt{s}$,
\begin{align}
\lambda_\star&=s-\gamma_*\sqrt{s},\\
\lambda_\bullet&=s-\gamma_\bullet\sqrt{s}=\lambda_\star+r_\bullet,
\end{align}
with
\begin{align}
r_\bullet=\frac{h_{\infty,F}^R(\gamma_*)}{g'_{\infty}(\gamma_*)}.
\end{align}
Table \ref{table:t1} presents results for the admission policy $p_k=p=0.1$, $k\geq s$, and for Problem 1 with $s=100$ servers.
Note that $|\lambda_{{\rm opt}}-\lambda_\bullet|$ is always less than $0.1$, and how the refinements $r_\bullet$ lead to much sharper estimates of the true optimal values.
\begin{table}[h]\centering
$
\begin{array}{ccccccc}
 \text{$\eps$} & \text{$\lambda_{\rm opt}$}  & \text{$\lambda_* $} & \text{$\lambda_\bullet$} & \text{$r_\bullet$} & \text{$\sqrt{s}D_F^R(s,\lambda_*)$} & \text{$\sqrt{s}D_F^R(s,\lambda_\bullet)$} \\ \hline
 \text{    0.010} & \text{   75.324} & \text{   72.836} & \text{   75.409} & \text{    2.573} & \text{    0.004} & \text{    0.010} \\
 \text{    0.020} & \text{   77.554} & \text{   75.504} & \text{   77.621} & \text{    2.117} & \text{    0.011} & \text{    0.020} \\
 \text{    0.030} & \text{   78.996} & \text{   77.201} & \text{   79.053} & \text{    1.852} & \text{    0.018} & \text{    0.030} \\
 \text{    0.040} & \text{   80.096} & \text{   78.479} & \text{   80.146} & \text{    1.667} & \text{    0.026} & \text{    0.041} \\
 \text{    0.050} & \text{   80.999} & \text{   79.519} & \text{   81.045} & \text{    1.525} & \text{    0.034} & \text{    0.051} \\
 \text{    0.060} & \text{   81.774} & \text{   80.405} & \text{   81.816} & \text{    1.411} & \text{    0.043} & \text{    0.061} \\
 \text{    0.070} & \text{   82.458} & \text{   81.181} & \text{   82.497} & \text{    1.315} & \text{    0.052} & \text{    0.071} \\
 \text{    0.080} & \text{   83.073} & \text{   81.876} & \text{   83.110} & \text{    1.234} & \text{    0.061} & \text{    0.081} \\
 \text{    0.090} & \text{   83.636} & \text{   82.507} & \text{   83.671} & \text{    1.164} & \text{    0.071} & \text{    0.091} \\
 \text{    0.100} & \text{   84.157} & \text{   83.088} & \text{   84.190} & \text{    1.102} & \text{    0.080} & \text{    0.101} \\
 \vspace{0.05mm}
\end{array}
$
 \caption{Results for the admission policy $p_k=p=0.1$, $k\geq s$, and for Problem 1 with $D=D_F^R$ and $s=100$ servers.}    \label{table:t1}
\end{table}

Table \ref{table:t1a} presents results for the same situation, except with $p_k=p=0.5$, $k\geq s$. This control policy is thus more lenient, in the sense that when all servers are occupied, 50\% (instead of 10\%) of the arrivals is admitted. Note that despite the fact that the refinements $r_\bullet$ are larger, and $|\lambda_{{\rm opt}}-\lambda_\bullet|$ is always less than 1, the estimates are less accurate. Based on this example, and many other numerical experiments not reported here, we conclude that square-root staffing becomes less accurate for systems with admission control policies that allow more customers to enter. Such systems really benefit from the refined staffing rules.

\begin{table}[h]\centering
$
\begin{array}{ccccccc}
 \text{$\eps$} & \text{$\lambda_{\rm opt}$}  & \text{$\lambda_* $} & \text{$\lambda_\bullet$} & \text{$r_\bullet$} & \text{$\sqrt{s}D_F^R(s,\lambda_*)$} & \text{$\sqrt{s}D_F^R(s,\lambda_\bullet)$} \\ \hline
 \text{    0.010} & \text{   75.910} & \text{   72.836} & \text{   76.298} & \text{    3.462} & \text{    0.003} & \text{    0.011} \\
 \text{    0.020} & \text{   78.162} & \text{   75.504} & \text{   78.510} & \text{    3.006} & \text{    0.009} & \text{    0.022} \\
 \text{    0.030} & \text{   79.619} & \text{   77.201} & \text{   79.942} & \text{    2.741} & \text{    0.015} & \text{    0.033} \\
 \text{    0.040} & \text{   80.730} & \text{   78.479} & \text{   81.035} & \text{    2.556} & \text{    0.022} & \text{    0.043} \\
 \text{    0.050} & \text{   81.642} & \text{   79.519} & \text{   81.933} & \text{    2.414} & \text{    0.029} & \text{    0.054} \\
 \text{    0.060} & \text{   82.425} & \text{   80.405} & \text{   82.705} & \text{    2.300} & \text{    0.037} & \text{    0.064} \\
 \text{    0.070} & \text{   83.116} & \text{   81.181} & \text{   83.386} & \text{    2.204} & \text{    0.045} & \text{    0.074} \\
 \text{    0.080} & \text{   83.738} & \text{   81.876} & \text{   83.999} & \text{    2.123} & \text{    0.053} & \text{    0.084} \\
 \text{    0.090} & \text{   84.307} & \text{   82.507} & \text{   84.560} & \text{    2.053} & \text{    0.061} & \text{    0.095} \\
 \text{    0.100} & \text{   84.832} & \text{   83.088} & \text{   85.078} & \text{    1.991} & \text{    0.070} & \text{    0.105} \\
 \vspace{0.05mm}
\end{array}
$
 \caption{Results for the admission policy $p_k=p=0.5$, $k\geq s$, and for Problem 1 with $s=100$ servers.}    \label{table:t1a}
\end{table}

Let us now turn to the same dimensioning problem, but then for the system {\it with} retrials. \\
{\bf Problem 2} For fixed $s,\eps$, find $\lambda$ such that
\beq \label{eqprob3}
\sqrt{s}D_F^R(s,\lambda+\Omega)=\eps
 \eq
with $\Omega$ defined as in \eqref{genC}.\\

%
%

%

We only consider the case $D_F^R$ in \eqref{eqprob3}, but results for $D_F$ can be obtained in a similar manner.
Let us first bring the generalized Cohen equation \eqref{genC} in a form that is amenable for analysis in the QED regime.
Write $\Omega=a\sqrt{s}$ so that \eqref{genC} becomes
\beq \label{5.2}
a\sqrt{s}=(s-(\gamma-a)\sqrt{s})\,D_F^R(s,s-(\gamma-a)\sqrt{s}).
\eq
With $f_{s,F}^R$ defined as
\beq \label{4.4}
f_{s,F}^R(\gamma):=\Bigl(1-\frac{\gamma}{\sqrt{s}}\Bigr)\,g_{s,F}^R(\gamma)=\sqrt{s} \Bigl(1-\frac{\gamma}{\sqrt{s}}\Bigr)\,D_F^R(s,s-\gamma\sqrt{s}),
\eq
we can write (\ref{5.2}) concisely as
\beq \label{5.3}
a=f_{s,F}^R(\gamma-a)
\eq
in which $\gamma$ is given and $a$ is to be solved.

%

We thus have $D_F^R(s,\lambda+\Omega)=\eps/\sqrt{s}$ if and only if
\beq \label{7.1}
g_{s,F}^R(\gamma-a_{s,F}(\gamma))=\eps,
\eq
where  $a=a_{s,F}(\gamma)$ solves for $\gamma>0$ the equation \eqref{5.3}.

Using $f_{s,F}^R(\delta)=(1-\delta/\sqrt{s})\,g_{s,F}^R(\delta)$, the equation in (\ref{7.1}) takes the form
\beq \label{7.4}
\frac{a}{1-(\gamma-a)/\sqrt{s}}=\eps,
\eq
where $a=a_{s,F}(\gamma)$. This gives
\beq \label{7.5}
a=\frac{1-\gamma/\sqrt{s}}{1-\eps/\sqrt{s}}\,\eps,
\eq
and therefore
\beq \label{7.6}
\gamma-a=\frac{\gamma-\eps}{1-\eps/\sqrt{s}}.
\eq
Hence, we can solve Problem 2 by first solving $\delta=\delta_{s,F}^R(\eps)$ from
\beq \label{thirty}
g_{s,F}^R(\delta)=\eps,
\eq
see (\ref{7.1}),
and then setting $(\gamma-\eps)/(1-\eps\sqrt{s})=\delta$, i.e.,
\beq \label{7.7}
\gamma=\gamma_{s,F}^R(\eps)=\delta+\eps-\delta\eps/\sqrt{s};\quad \delta=\delta_{s,F}^R(\eps).
\eq
As to \eqref{thirty} we operate under the assumption of monotonicity of $D_F^R$ as made in the beginning of this section.

Denote by $\delta=\delta_\infty(\eps)$ the solution of
\beq
g_\infty(\delta)=\eps.
\eq
Observe that $\delta_\infty(\eps)=\gamma_\infty^R(\eps)$, see \eqref{defdd}. Then we have by Theorem \ref{thmasss}, \eqref{7.5} and \eqref{7.6} that
\beq
\delta_{s,F}^R(\eps)=\delta_\infty(\eps)-\frac{1}{\sqrt{s}}\frac{h_{\infty,F}^R(\delta_\infty(\eps))}{g_\infty'(\delta_\infty(\eps))}+O(s^{-1}).
\eq
Using this in \eqref{7.7}, we arrive at the following result.

\begin{thm}[Asymptotic dimensioning with retrials]\label{thmasss2} Under condition \eqref{condA}, and assuming that $D_F^R(s,s-\gamma\sqrt{s})$ is monotonic in $\gamma$,
\begin{align}
\gamma_{s,F}^R(\eps)&=\eps+\delta_\infty(\eps)+\frac{1}{\sqrt{s}}\theta_{\infty,F}^R(\eps)+O(s^{-1}),
\end{align}
with
\begin{align}
\theta_{\infty,F}^R(\eps)&=-\delta_\infty(\eps)\eps-\frac{h_{\infty,F}^R(\delta_{\infty}(\eps))}{g'_{\infty}(\delta_{\infty}(\eps))}.
\end{align}
\end{thm}

For the asymptotic dimensioning scheme, applied to Problem 1 {\it with retrials}, we identify the following key functions and  parameters:
\begin{align}
\gamma_*&=\eps+\delta_\infty(\eps),\\
\gamma_\bullet&=\gamma_* + \frac{1}{\sqrt{s}}\theta_{\infty,F}^R(\eps),
\end{align}
and using the square-root rule $\lambda=s-\gamma\sqrt{s}$,
\begin{align}
\lambda_\star&=s-\gamma_*\sqrt{s},\\
\lambda_\bullet&=s-\gamma_\bullet\sqrt{s}=\lambda_\star+r_\bullet,
\end{align}
with
\begin{align}
r_\bullet=(\gamma_*-\eps)\eps+\frac{h_{\infty,F}^R(\gamma_*-\eps)}{g'_{\infty}(\gamma_*-\eps)}.
\end{align}
Table \ref{table:t2} presents results for the admission policy $p_k=p=0.1$, $k\geq s$, and for Problem 1 with retrials and $s=100$ servers. Table  \ref{table:t2b} displays the results for $p=0.5$.

\begin{table}[h]\centering
$
\begin{array}{ccccccc}
 \text{$\eps$} & \text{$\lambda_{\rm opt}$}  & \text{$\lambda_* $} & \text{$\lambda_\bullet$} & \text{$r_\bullet$} & \text{$\sqrt{s}D_F^R(s,\lambda_*)$} & \text{$\sqrt{s}D_F^R(s,\lambda_\bullet)$} \\ \hline
 \text{    0.010} & \text{   75.249} & \text{   72.736} & \text{   75.336} & \text{    2.600} & \text{    0.004} & \text{    0.010} \\
 \text{    0.020} & \text{   77.399} & \text{   75.304} & \text{   77.470} & \text{    2.166} & \text{    0.010} & \text{    0.020} \\
 \text{    0.030} & \text{   78.759} & \text{   76.901} & \text{   78.822} & \text{    1.921} & \text{    0.018} & \text{    0.031} \\
 \text{    0.040} & \text{   79.775} & \text{   78.079} & \text{   79.832} & \text{    1.753} & \text{    0.025} & \text{    0.041} \\
 \text{    0.050} & \text{   80.594} & \text{   79.019} & \text{   80.647} & \text{    1.628} & \text{    0.034} & \text{    0.051} \\
 \text{    0.060} & \text{   81.283} & \text{   79.805} & \text{   81.333} & \text{    1.528} & \text{    0.042} & \text{    0.061} \\
 \text{    0.070} & \text{   81.880} & \text{   80.481} & \text{   81.929} & \text{    1.447} & \text{    0.051} & \text{    0.071} \\
 \text{    0.080} & \text{   82.409} & \text{   81.076} & \text{   82.455} & \text{    1.379} & \text{    0.059} & \text{    0.081} \\
 \text{    0.090} & \text{   82.884} & \text{   81.607} & \text{   82.929} & \text{    1.321} & \text{    0.068} & \text{    0.091} \\
 \text{    0.100} & \text{   83.315} & \text{   82.088} & \text{   83.359} & \text{    1.271} & \text{    0.077} & \text{    0.101} \\
 \vspace{0.05mm}
\end{array}
$
 \caption{Results for the admission policy $p_k=p=0.1$, $k\geq s$, and for Problem 2 with retrials and $s=100$ servers.}    \label{table:t2}
\end{table}

\begin{table}[h]\centering
$
\begin{array}{ccccccc}
 \text{$\eps$} & \text{$\lambda_{\rm opt}$}  & \text{$\lambda_* $} & \text{$\lambda_\bullet$} & \text{$r_\bullet$} & \text{$\sqrt{s}D_F^R(s,\lambda_*)$} & \text{$\sqrt{s}D_F^R(s,\lambda_\bullet)$} \\ \hline
 \text{    0.010} & \text{   75.834} & \text{   72.736} & \text{   76.225} & \text{    3.489} & \text{    0.003} & \text{    0.011} \\
 \text{    0.020} & \text{   78.006} & \text{   75.304} & \text{   78.359} & \text{    3.055} & \text{    0.009} & \text{    0.022} \\
 \text{    0.030} & \text{   79.380} & \text{   76.901} & \text{   79.711} & \text{    2.810} & \text{    0.015} & \text{    0.033} \\
 \text{    0.040} & \text{   80.407} & \text{   78.079} & \text{   80.721} & \text{    2.642} & \text{    0.021} & \text{    0.043} \\
 \text{    0.050} & \text{   81.234} & \text{   79.019} & \text{   81.536} & \text{    2.516} & \text{    0.028} & \text{    0.054} \\
 \text{    0.060} & \text{   81.930} & \text{   79.805} & \text{   82.222} & \text{    2.417} & \text{    0.036} & \text{    0.064} \\
 \text{    0.070} & \text{   82.534} & \text{   80.481} & \text{   82.817} & \text{    2.336} & \text{    0.043} & \text{    0.074} \\
 \text{    0.080} & \text{   83.068} & \text{   81.076} & \text{   83.344} & \text{    2.268} & \text{    0.051} & \text{    0.085} \\
 \text{    0.090} & \text{   83.548} & \text{   81.607} & \text{   83.817} & \text{    2.210} & \text{    0.059} & \text{    0.095} \\
 \text{    0.100} & \text{   83.984} & \text{   82.088} & \text{   84.248} & \text{    2.160} & \text{    0.067} & \text{    0.105} \\
 \vspace{0.05mm}
\end{array}
$
 \caption{Results for the admission policy $p_k=p=0.5$, $k\geq s$, and for Problem 2 with retrials and $s=100$ servers.}    \label{table:t2b}
\end{table}

\section{Bistability}\label{sec:stab}
In Section \ref{sec:dim} we have considered dimensioning problems that were formulated directly in
terms of the probabilities $D_F$ and $D_F^R$, and we have seen that these problems admit generically one solution when the prescribed value for $D_F$ or $D_F^R$ is in the appropriate range. The situation is different when dimensioning problems of a more complicated nature are considered. In this section we consider carried traffic quantities $\lambda(1-D(s,\lambda))$ with $D=D_F$ or $D_F^R$, without and with retrials, and we will see that the corresponding dimensioning problems may have two solutions. One could refer to this situation as {\it bistability}, which has also been observed in certain loss networks with alternative two-link routing \cite{bistability}. The bistability stresses the fact that sample paths of the Markov process tend to be concentrated around two relatively stable points of the system.

 Note that $\lambda(1-D_F(s,\lambda))$ is the rate of arrivals that enter the system in one attempt, and $\lambda(1-D_F^R(s,\lambda))$  is the rate of arrivals that pass the system, possibly after having to wait in the queue that arises when all servers are occupied.

 We first consider the system without retrials. \\
\noindent {\bf Problem 3}
For fixed $s,\eps$, find $\lambda$ such that
\beq \label{eqii} \lambda(1-D(s,\lambda))=\eps \quad {\rm with} \quad D=D_F \ {\rm or} \ D_F^R.
\eq
In Section \ref{sec4} we prove the following result.

\begin{thm}[(Non-)uniqueness of solutions]\label{thm12} Under condition \eqref{condA},\\
{\rm (i)} Equation \eqref{eqii} with $D=D_F$ has at least two solutions $\lambda\in(0,\lambda_P)$ when $\eps>0$ is sufficiently small.\\
{\rm (ii)}  Equation \eqref{eqii} with $D=D_F^R$ has a unique solution $\lambda\in(0,\lambda_P)$  when $\eps\in(0,s)$.
\end{thm}

Theorem \ref{thm12}(i) is intuitively clear since $\lambda(1-D_F(s,\lambda))$  is positive for all $\lambda\in (0,\lambda_P)$ while it vanishes when $\lambda\downarrow 0$ or $\lambda\uparrow \lambda_P$. The result of Theorem \ref{thm12}(ii) is less obvious and depends on monotonicity of $\lambda(1-D_F^R(s,\lambda))$ as a function of $\lambda\in (0,\lambda_P)$.
For the case $p_k=p\in(0,1)$, $k\geq s$, it follows from \eqref{622} and Proposition~\ref{propp3}(iv) that
$\lambda(1-D_F(s,\lambda))$ is strictly concave, and so \eqref{eqii} with $D=D_F$ has {\it exactly} two solutions when $\eps>0$ is sufficiently small.
%

We next consider carried traffic quantities in the case of retrials. For convenience, we shall restrict ourselves to the choice $D=D_F^R$ in \eqref{eqii}.

\noindent {\bf Problem 4}
For fixed $s,\eps$, find $\lambda$ such that
\beq \label{5222}
\lambda(1-D_F^R(s,\lambda+\Omega))=\eps
 \eq
 with $\Omega$ defined as in \eqref{genC}.

Observe that $\lambda(1-D_F^R(s,\lambda+\Omega))=\eps$ if and only if
\beq \label{7.2}
(s-\gamma\sqrt{s})\Bigl(1-\frac{1}{\sqrt{s}}\,g_{s,F}^R(\gamma-a_{s,F}(\gamma))\Bigr)=\eps.
\eq
Using $f_{s,F}^R(\delta)=(1-\delta/\sqrt{s})\,g_{s,F}^R(\delta)$ and
(\ref{5.3}) in (\ref{7.2}), we can write equation \eqref{5222} as
\beq \label{7.8}
(s-\gamma\sqrt{s})\Bigl(1-\frac{1}{\sqrt{s}}\:\frac{a}{1-(\gamma-a)/\sqrt{s}}\Bigr)=\sqrt{s}\:\frac{(\sqrt{s}-\gamma)^2} {\sqrt{s}-\gamma+a}=\eps,
\eq
where $a=a_{s,F}(\gamma)$. It is now not possible to simply eliminate $a$, and a study of the function
\beq \label{7.9}
L_{s,F}(\gamma)=\frac{(\sqrt{s}-\gamma)^2}{\sqrt{s}-\gamma+a_{s,F}(\gamma)}
\eq
is required.

We now give a detailed presentation of what can be achieved analytically for the case that $F\equiv0$ ($p_k=0$ for all $k\geq s$). It is a challenging problem to explore in what respect this detailed analytic result can be extended to more general admission policies.
We thus consider the carried-traffic problem in (\ref{7.2}) for the case that $F\equiv 0$ and study the function
\beq \label{7.10}
L_s(\gamma)=L_{s,F}(\gamma)=\frac{(\sqrt{s}-\gamma)^2}{\sqrt{s}-\gamma+a_s(\gamma)},\quad 0<\gamma<\sqrt{s},
\eq
where $a_s(\gamma)$ is the solution of the Cohen equation $a=f_s(\gamma-a)$ with $f_s=f_{s,F\equiv 0}$. Both $f_s$ and $a_s$ have been studied in great detail in \cite{ajl}. Using the results of \cite{ajl} and extensions thereof, the following is shown in Appendix~\ref{aB}.
\begin{thm}\label{thmm3}There holds
\begin{align} \label{7.11}
L_s(\gamma)&\in(0,\sqrt{s}-\gamma), \quad \gamma\in(0,\sqrt{s}), \\
\label{7.12}
L_s(\gamma)&=\gamma\,s(1+O(\gamma\sqrt{s})),\quad \gamma\downarrow0,\\
 \label{7.13}
L_s(\gamma)&=(\sqrt{s}-\gamma)\Bigl(1+O\Bigl(\frac{e^s}{\sqrt{s}}\,\Bigl(1-\frac{\gamma}{\sqrt{s}}\Bigr)^s \Bigr)\Bigr),\quad \gamma\uparrow\sqrt{s}.
\end{align}
Furthermore, $L_s(\gamma)$ is unimodal on $(0,\sqrt{s})$, and the maximum of $L_s(\gamma)$ is assumed at the unique solution $\gamma=\hat{\gamma}_s$ of the equation
\beq \label{7.14}
\gamma\,a_s(\gamma)=\frac12\Bigl(1-\frac{\gamma}{\sqrt{s}}\Bigr),
\eq
and
\beq \label{7.15}
L_s(\hat{\gamma}_s)=\frac{\hat{\gamma}_s}{\hat{\gamma}_s+\dfrac{1}{2\sqrt{s}}}\,(\sqrt{s}-\hat{\gamma}_s).
\eq
Finally,
\beq \label{7.16}
\hat{\gamma}_s> f_s(0)=\gamma_s^{\ast}>\Bigl(\frac12+\frac{1}{16s}\Bigr)^{1/2}-\frac{1}{4\sqrt{s}},~~~~~~ s\geq2,
\eq
and $\hat{\gamma}_s$ increases in $s$ from $1/2$ at $s=1$ to $1.034113461$ at $s=\infty$, with $\hat{\gamma}_s\approx 1$ when $s=550$.
\end{thm}
From the detailed information provided by Theorem~\ref{thmm3}, it is seen that the equation
\beq \label{7.17}
\frac{(\sqrt{s}-\gamma)^2}{\sqrt{s}-\gamma+a}=\frac{\eps}{\sqrt{s}},
\eq
see (\ref{7.8}), has two, one or zero solutions according as $\eps$ $<$, $=$ or $>$ $\sqrt{s}\,L_s(\hat{\gamma}_s)$ with $L_s(\hat{\gamma}_s)$ given in (\ref{7.15}). Furthermore, it is seen that $\sqrt{s}\,L_s(\hat{\gamma}_s)=s+O(\sqrt{s})$, and that the highest carried-traffic numbers occur in a $\gamma$-region somewhat away from, but relatively close to, $\gamma=0$.

\begin{thm}[Bistability]\label{thmbi}
Problem $4$ has two, one or zero solutions according as $\eps$ $<$, $=$ or $>$ $\sqrt{s}\,L_s(\hat{\gamma}_s)$ with $L_s(\hat{\gamma}_s)$ given in {\rm(\ref{7.15})}.
\end{thm}

In Figure~\ref{fig2} we display $\frac{1}{\sqrt{s}}\,L_s(\delta\sqrt{s})$, $0<\delta<1$, for $s=1,5,10,50,100$. It is observed that the graphs approximate the graph of the function $1-\delta$, $0<\delta<1$, when $s$ gets large.

\begin{figure}[hbtp]
\begin{center}
 \includegraphics[width= .5\linewidth]{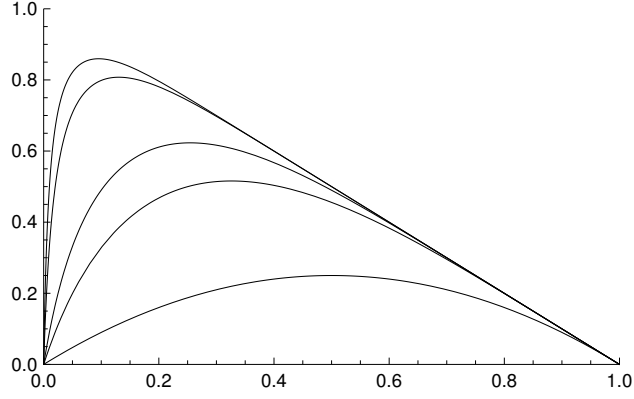}
  \end{center}
\caption{Plot of $\frac{1}{\sqrt{s}}\,L_s(\delta\sqrt{s})$, $0<\delta<1$, for the cases $s=1,5,10,50,100\,$, with maximizer $\hat{\gamma}_s/\sqrt{s}$ decreasing with $s$ and maximum value $\frac{1}{\sqrt{s}}\,L_s(\hat{\gamma}_s)$ increasing with $s$.\label{fig2}}
\end{figure}

\section{Proof of results on dimensioning} \label{sec4}
In this section we drop condition \eqref{condA} so that both $P=1$ and $F(\lambda_P-0)<\infty$ are allowed.

The equations in \eqref{eqi} and in \eqref{eqii} take the form
\beq \label{4.5}
g(\gamma)=\eps,~~~~~~g=g_{s,F}~{\rm or}~g_{s,F}^R,
\eq
and
\beq \label{4.6}
\gamma+f(\gamma)=\frac{s-\eps}{\sqrt{s}}, \quad f=f_{s,F} \ {\rm or} \ f_{s,F}^R,
\eq
respectively, where $g_{s,F}^R$ and $f_{s,F}^R$ as in \eqref{4444} and \eqref{4.4}, and
we define for $\gamma\in (\gamma_P,\sqrt{s}]$,
\beq \label{4.3}
f_{s,F}(\gamma):=\Bigl(1-\frac{\gamma}{\sqrt{s}}\Bigr)\,g_{s,F}(\gamma)=
\sqrt{s}\Bigl(1-\frac{\gamma}{\sqrt{s}}\Bigr)\,D_F(s,s-\gamma\sqrt{s}).
\eq

To get insight and pertinent results about when the equations in (\ref{4.5}) and (\ref{4.6}) have (unique) solutions and how to find these solutions, we relate the functions $f_{s,F}$ and $f_{s,F}^R$ to the functions $f_s$ and $g_s$ given by
\beq \label{4.7}
f_s(\gamma)=\Bigl(1-\frac{\gamma}{\sqrt{s}}\Bigr)\,g_s(\gamma)=\sqrt{s}\Bigl(1-\frac{\gamma}{\sqrt{s}}\Bigr)\, B(s,s-\gamma\sqrt{s}),~~~~~~\gamma\leq\sqrt{s}.
\eq
The latter functions have been studied in considerable detail in \cite{ajl}, in particular with respect to monotonicity properties. Some of these properties are collected in the beginning of Appendix \ref{aA}.
\begin{lem}\label{lemm1} For $\gamma\in(\gamma_P,\sqrt{s}]$,
\beq \label{4.8}
f_{s,F}(\gamma)=f_s(\gamma)\,\frac{1+\Bigl(1-\dfrac{\gamma}{\sqrt{s}}\Bigr)\,H_s(\gamma)} {1+\dfrac{1}{\sqrt{s}}\,f_s(\gamma)\,H_s(\gamma)},
\eq
and
\beq \label{4.9}
f_{s,F}^R(\gamma)=f_s(\gamma)\,\frac{1-\dfrac{\gamma}{\sqrt{s}}\,H_s(\gamma)} {1+\dfrac{1}{\sqrt{s}}\,f_s(\gamma)\, H_s(\gamma)},
\eq
with
\eqan{\label{4.10}
H_s(\gamma)=\sum_{n=0}^\infty p_s\cdot \cdots p_{s+n}\Big(1-\frac{\gamma}{\sqrt{s}}\Big)^n.
}
\end{lem}
\begin{proof}This follows in a straightforward manner from (\ref{3.4}) and (\ref{3.5}) and the definition of $H_s$, where we note that $F(1-\frac{\gamma}{\sqrt{s}})=(1-\frac{\gamma}{\sqrt{s}})\,H_s(\gamma)$.
\end{proof}

Lemma~\ref{lemm1} shows that $f_{s,F}$ and $f_{s,F}^R$ factorize into an admission policy independent part $f_s(\gamma)$ and a part comprising the admission policy via $H_s$. By dividing either side of (\ref{4.8}) and (\ref{4.9}) by $(1-\gamma/\sqrt{s})$, it is seen that a similar result as Lemma~\ref{lemm1} holds for $g_{s,F}$ and $g_{s,F}^R$, with the same factors comprising the admission policy as in (\ref{4.8}) and (\ref{4.9}).

The following result gives a global picture for $f_{s,F}$ and $f_{s,F}^R$ in terms of inequalities. We observe that $F(\lambda_P-0)=\infty \Leftrightarrow H_s(\gamma_P+0)=\infty$.
\begin{prop}\label{propp1}
\bi{(vii)0}
\ITEM{(i)} For $\gamma_P<\gamma<\sqrt{s}$,
\beq \label{4.11}
\max\,\{0,{-}\gamma\}\leq f_{s,F}^R(\gamma)\leq f_s(\gamma)\leq f_{s,F}(\gamma)\leq\min\Bigl\{\sqrt{s}-\gamma,\frac{\sqrt{s}\,f_s(\gamma)} {\gamma+f_s(\gamma)}\Bigr\}.
\eq
\ITEM{(ii)} There is equality in the first inequality in {\rm (\ref{4.11})} if and only if $\gamma\in(0,\sqrt{s})$ and
$p_k=1$ for all $k\geq s$, and in that case
\beq \label{4.12}
H(\gamma)=\frac{\sqrt{s}}{\gamma},~~~~~~\gamma_P=0<\gamma<\sqrt{s}.
\eq
\ITEM{(iii)} There is equality in the second inequality in {\rm (\ref{4.11})} for any $\gamma\in(\gamma_P,\sqrt{s})$ if and only if $p_k=0$ for all $k\geq s$. There is equality in the third inequality in {\rm (\ref{4.11})} for any $\gamma\in(\gamma_P,\sqrt{s})$ if and only if $p_k=0$ for all $k\geq s$.
\ITEM{(iv)} There is equality in the fourth inequality in {\rm (\ref{4.11})} if and only if $\gamma\in(0,\sqrt{s})$ and $p_k=1$ for all $k\geq s$.
\ITEM{(v)} For $\gamma=\sqrt{s}$,
\beq \label{4.13}
f_{s,F}^R(\gamma)=f_s(\gamma)=f_{s,F}(\gamma)=0.
\eq
\ITEM{(vi)} $f_{s,F}^R(\gamma_P+0)={-}\gamma_P$ if and only if $H_s(\gamma_P+0)=\infty$.
\ITEM{(vii)}  $f_{s,F}(\gamma_P+0)=\sqrt{s}-\gamma_P$ if and only if $H_s(\gamma_P+0)=\infty$.
\ei
\end{prop}

The proof of Proposition \ref{propp1} is given in Appendix \ref{aA} and uses the representations (\ref{4.8}) and (\ref{4.9}).

\begin{figure}[hbtp]
\begin{center}
 \includegraphics[width= .5\linewidth]{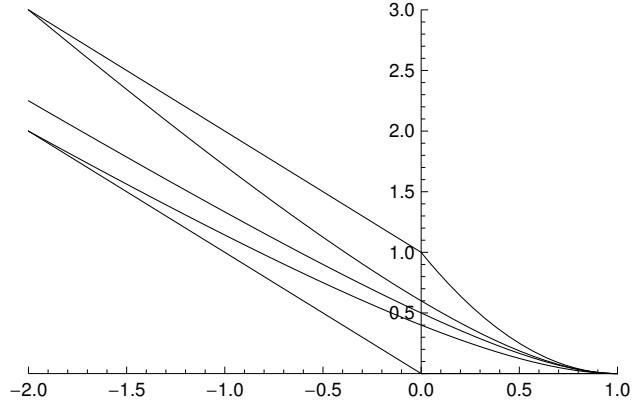}
  \end{center}
\caption{Picture illustrating $f_{s,F}^R$, $f_{s,F}$ and $f_s$ with bounds as in Proposition~\ref{propp1} for the case that $0<P<1$ and $H_s(\gamma_P+0)=\infty$ (actual choice: $s=1$, $p_k=p=1/3$, $k\geq s$)\label{fig1}}
\end{figure}

In Figure~\ref{fig1} we show $f_{s,F}^R$, $f_s$ and $f_{s,F}$, together with the bounds in {\rm (\ref{4.11})} for the case that ${-}\infty<\gamma_P<0$ and $H(\gamma_P+0)=\infty$ (actual choice: $s=1$ and $p_k=p=1/3$, $k\geq s$).

In general, monotonicity properties for the functions $f_{s,F}$, $g_{s,F}^R$ and of the functions $\gamma+f_{s,F}(\gamma)$ and $\gamma+f_{s,F}^R(\gamma)$, see (\ref{4.5}) and (\ref{4.6}) are not easy to establish or manifestly not true. We now present some results, positive and negative, on monotonicity and convexity of the functions in (\ref{4.3}) and (\ref{4.4}).
\begin{prop}\label{propp2}
\bi{(ii)0}
\ITEM{(i)} $f_{s,F}$ and $g_{s,F}$ are strictly decreasing.
\ITEM{(ii)} $\gamma+f_{s,F}^R(\gamma)$, $\gamma_P<\gamma\leq\sqrt{s}$, is strictly increasing.
\ei
\end{prop}
The proof of Proposition \ref{propp2} is given in Appendix~\ref{aA}.
An interesting example is obtained when we choose
\beq \label{4.14}
p_k=p\in(0,1), \quad k\geq s.
\eq
Then $P=p$, $\gamma_P={-}\frac{1-p}{p}\,\sqrt{s}\in({-}\infty,0)$, and
\beq \label{4.15}
H_s(\gamma)=\sum_{n=0}^{\infty}\,p^{n+1}\Bigl(1-\frac{\gamma}{\sqrt{s}}\Bigr)^n= \frac{p}{1-p(1-\gamma/\sqrt{s})} ,~~~~~\gamma_P<\gamma\leq\sqrt{s}.
\eq
Furthermore, $H_s(\gamma_P+0)=\infty$, and we compute from (\ref{4.8}), (\ref{4.9}) and (\ref{4.15})
\beq \label{4.16}
f_{s,F}(\gamma)=\frac{f_s(\gamma)}{1-p+\dfrac{p}{\sqrt{s}}\,(\gamma+f_s(\gamma))},~~~~~ f_{s,F}^R(\gamma)= \frac{(1-p)\,f_s(\gamma)}{1-p+\dfrac{p}{\sqrt{s}}\,(\gamma+f_s(\gamma))}.
\eq
There is for this case the following result.
\begin{prop}\label{propp3}
Assume $p_k=p\in(0,1)$, $k\geq s$, so that $f_{s,F}$ and $f_{s,F}^R$ are given by {\rm(\ref{4.16})}, and let $g_{s,F}$ and $g_{s,F}^R$ be the functions associated to $f_{s,F}$ and $f_{s,F}^R$ according to {\rm(\ref{4.3})} and {\rm(\ref{4.4})}, respectively. Then,
\bi{(iii)0}
\ITEM{(i)} $f_{s,F}$ and $f_{s,F}^R$ are strictly decreasing and convex in $\gamma_P<\gamma\leq\sqrt{s}$, with
\beq \label{4.17}
f_{s,F}(\gamma_P+0)=\sqrt{s}-\gamma_P=\frac{\sqrt{s}}{p},~~~~~~f_{s,F}^R(\gamma_P+0)={-}\gamma_P=\frac{1-p}{p}\, \sqrt{s},
\eq
and $f_{s,F}(\sqrt{s})=f_{s,F}^R(\sqrt{s})=0$.
\ITEM{(ii)} $g_{s,F}$ and $g_{s,F}^R$ are strictly decreasing in $\gamma_P<\gamma\leq\sqrt{s}$, with
\beq \label{4.18}
g_{s,F}(\gamma_P+0)=\sqrt{s},~~~~~~g_{s,F}^R(\gamma_P+0)=(1-p)\,\sqrt{s},
\eq
and $g_{s,F}(\sqrt{s})=g_{s,F}^R(\sqrt{s})=0$.
\ITEM{(iii)} $\gamma+f_{s,F}^R(\gamma)$ is strictly increasing and convex in $\gamma_P<\gamma\leq\sqrt{s}$, with
\beq \label{4.19}
\Bigl.\gamma+f_{s,F}^R(\gamma)\Bigr|_{\gamma=\gamma_P+0}=0,~~~~~~\Bigl.\gamma+f_{s,F}^R(\gamma)\Bigr|_{\gamma=\sqrt{s}}
=\sqrt{s}.
\eq
\ITEM{(iv)} $\gamma+f_{s,F}(\gamma)$ is non-monotonic and strictly convex in $\gamma_P<\gamma\leq\sqrt{s}$, with
\beq \label{4.20}
\Bigl.\gamma+f_{s,F}(\gamma)\Bigl|_{\gamma=\gamma_P+0}=\Bigl.\gamma+f_{s,F}(\gamma)\Bigr|_{\gamma=\sqrt{s}} =\sqrt{s}.
\eq
\ei
\end{prop}
The proof of Proposition \ref{propp3} is given in Appendix~\ref{aA}.

We finally have the following result.
\begin{prop}\label{propp6}
Assume that $p_k$ is chosen as in \eqref{4.3a}. Then  $g_{s,F}^R(\gamma)$ is strictly decreasing  in $\gamma\in(\gamma_P,\sqrt{s})$.
\end{prop}
The proof of Proposition \ref{propp6} is given in the Appendix \ref{aA}.

We now show how the various results given in this section can be used to prove Theorem \ref{thm7} and Theorem \ref{thm12}.
\subsection{Proof of Theorem \ref{thm7}}
Under condition \eqref{condA} we have $F(\lambda_P-0)=H_s(\gamma_P+0)=\infty$. It follows from \eqref{4.3} and  \eqref{4.8} that
\beq\label{618}
g_{s,F}(\gamma_P+0)=\frac{f_{s,F}(\gamma_P+0)}{1-\gamma_P/\sqrt{s}}=\frac{f_{s}(\gamma_P)}{1-\gamma_P/\sqrt{s}}\frac{1-\gamma_P/\sqrt{s}}{f_{s}(\gamma_P)/\sqrt{s}}=\sqrt{s}.
\eq
In a similar fashion it follows from \eqref{4444}, \eqref{4.4} and \eqref{4.9}, using $\gamma_P=-(1-P)P^{-1}\sqrt{s}$, that
\beq\label{619}
g_{s,F}^R(\gamma_P+0)=(1-P)\sqrt{s}.
\eq
Furthermore,
\beq\label{620}
D_F(s,\lambda)=\sqrt{s}g_{s,F}(\gamma), \quad D_F^R(s,\lambda)=\sqrt{s}g_{s,F}^R(\gamma)
\eq
 when $\lambda=s-\gamma\sqrt{s}$ while $g_{s,F}(\sqrt{s})=g_{s,F}^R(\sqrt{s})=0$ as $B(s,0)=0$. Then Theorem \ref{thm7}(i) follows from Proposition \ref{propp2}(i) while Theorem \ref{thm7}(ii) holds by the monotonicity assumption made in the discussion preceding Theorem \ref{thm7}.

\subsection{Proof of Theorem \ref{thm12}}
The function $\lambda(1-D_F(s,\lambda))$ depends continuously on $\lambda\in(0,\lambda_P)$, is positive for
$\lambda\in(0,\lambda_P)$, and satisfies
\beq\label{621}
\lim_{\lambda\downarrow 0}\lambda(1-D_F(s,\lambda))=0=\lim_{\lambda\uparrow \lambda_P}\lambda(1-D_F(s,\lambda)),
\eq
see \eqref{618} and \eqref{620}. This yields assertion (i).

A computation, using \eqref{4.4} and \eqref{4.3}, shows that
\beq\label{622}
\lambda(1-D_F(s,\lambda))=s-\sqrt{s}(\gamma+f_{s,F}(\gamma)),
\eq
when $\lambda=s-\gamma\sqrt{s}$. Hence, by Proposition \ref{propp2}(ii), we have that $\lambda(1-D_F^R(s,\lambda))$
is strictly increasing in $\lambda\in(0,\lambda_P)$. Furthermore,
\beq\label{623}
\lim_{\lambda\downarrow 0}\lambda(1-D_F^R(s,\lambda))=0,\quad \lim_{\lambda\uparrow \lambda_P}\lambda(1-D_F^R(s,\lambda))=s
\eq
by \eqref{618} and \eqref{620}. This yields assertion (ii).

Let us conclude this section with the following two observations. Concavity of
$\lambda(1-D(s,\lambda))$ when $p_k=p\in(0,1)$, $k\geq s$, follows from \eqref{622} and Proposition \ref{propp3}(ii).
Monotonicity of $D_F^R(s,\lambda)$ when $p_k$ is given by \eqref{4.3} follows from \eqref{620} and Proposition \ref{propp6}.

\section{Proof of Theorem \ref{thm:qedretrials}} \label{sec6}
\mbox{} \\[-9mm]
We prove Theorem \ref{thm:qedretrials} under condition \eqref{condA}, where we assume $P\in(0,1)$
(the case $P=0$ requires only minor modification of the analysis below).
By Theorem \ref{thmm1},
the generalized Cohen equation \eqref{genC}, written in QED coordinates as in \eqref{5.3}, has a unique solution $a_{s, F}(\gamma)$ for any $\gamma\in(0,\sqrt{s})$.

%
\begin{thm}\label{thmm2}For any $\gamma>0$,
\beq \label{6.3}
a_{\infty}(\gamma)\geq a_{s,F}(\gamma)=a_{\infty}(\gamma)+O\Bigl(\frac{1}{\sqrt{s}}\Bigr),~~~~~~s\pr\infty,
\eq
where the $O$ holds uniformly in any compact set of $\gamma\in(0,\infty)$, and where $a_{\infty}(\gamma)$ is the unique solution of
\beq \label{6.4}
a=f_{\infty}(\gamma-a)=\frac{\varp(\gamma-a)}{\Phi(\gamma-a)},
\eq
see {\rm \cite[Section~3.3]{ajl}}.
\end{thm}
\begin{proof}We have $f_{s,F}^R(\delta)\leq f_s(\delta)<f_{\infty}(\delta)$ and $f_{s,F}^R(\delta)=f_{\infty}(\delta)+O(\frac{1}{\sqrt{s}})$ uniformly in any compact set of $\delta\in\dR$, see (\ref{4.9}), Proposition~\ref{propp1}(i) and \cite[Propositions 5 and 6]{ajl}. Let $\gamma>0$.
We prove below that $a_{s,F}(\gamma)\leq a_{\infty}(\gamma)$. Therefore, there is an $M>0$ such that
\beq \label{6.5}
f_{\infty}(\gamma-a)-\frac{M}{\sqrt{s}}\leq f_{s,F}^R(\gamma-a)\leq f_{\infty}(\gamma-a)
\eq
holds for $a\in[0,a_{\infty}(\gamma)]$ and all $s\geq 1$. Since $a+f_{\infty}(\gamma-a)$ is convex in $a$ and $f_{\infty}'(\delta)>{-}1$, $\delta\in\dR$, see \cite[(21)]{ajl}, it follows from the mean value theorem that there is an $\eta>0$ such that
\beq \label{6.6}
f_{\infty}(\gamma-a)-a\geq\eta(a_{\infty}(\gamma)-a),~~~~~~a\in[0,a_{\infty}(\gamma)],
\eq
where it is observed that $f_{\infty}(\gamma-a)-a=0$ at $a=a_{\infty}(\gamma)$. Then
\begin{eqnarray} \label{6.7}
0=f_{s,F}^R(\gamma-a_{s,F}(\gamma))-a_{s,F}(\gamma) & \geq & f_{\infty}(\gamma-a_{s,F}(\gamma))-a_{s,F}(\gamma)- \frac{M}{\sqrt{s}} \nonumber \\[3mm]
& \geq & \eta(a_{\infty}(\gamma)-a_{s,F}(\gamma))-\frac{M}{\sqrt{s}},
\end{eqnarray}
and so
\beq \label{6.8}
a_{s,F}(\gamma)\geq a_{\infty}(\gamma)-\frac{M}{\eta\sqrt{s}}.
\eq
In the above argument, $M$ and $\eta$ can be chosen independently of $\gamma$ in any compact subset of $(0,\infty)$, and so the result follows.

We still have to show that  $a_{s,F}(\gamma)\leq a_{\infty}(\gamma)$. We have from $f_{s,F}^R(\delta)\leq f_s(\delta)$ that
\beq
f_\infty(\gamma-a_{s,F}(\gamma))\geq f_{s,F}^R(\gamma-a_{s,F}(\gamma))=a_{s,F}(\gamma).
\eq
Therefore,
\beq\label{78}
\gamma-a_\infty(\gamma)+f_\infty(\gamma-a_{\infty}(\gamma))=\gamma\geq  \gamma-a_{s,F}(\gamma)+f_\infty(\gamma-a_{s,F}(\gamma)).
\eq
Now $\delta+f_\infty(\delta)$ is increasing, see \cite[(21)]{ajl}, and so it follows from \eqref{78} that $\gamma-a_\infty(\gamma)\geq \gamma-a_{s,F}(\gamma)$, i.e.,~$a_{s,F}(\gamma)\leq a_{\infty}(\gamma)$.
\end{proof}

The limiting behavior of $D_F(s,\lambda+\Omega)$ and $D_F^R(s,\lambda+\Omega)$  is now easily found. We have, see \eqref{4.3} and \eqref{4999a}, for $\gamma>0$
\begin{eqnarray} \label{6.9}
& \mbox{} & \hspace*{-1.4cm}\frac{1}{\sqrt{s}}\,D_F^{-1}(s,s-(\gamma-a_{s,F}(\gamma)))=\frac{1-(\gamma-a_{s,F}(\gamma))/\sqrt{s}} {f_{s,F}(\gamma-a_{s,F}(\gamma))}\nonumber \\[3mm]
& & \hspace*{-1.4cm}\pr~\Bigl(1+H(0)\,\frac{\phi(\gamma-a_{\infty}(\gamma))}{\Phi(\gamma-a_{\infty}(\gamma))}\Bigr)^{-1}= ((1+H(0))\,a_{\infty}(\gamma))^{-1},~~~~s\pr\infty\,.
\end{eqnarray}
Similarly, 
\begin{eqnarray} \label{6.10}
& \mbox{} & \frac{1}{\sqrt{s}}\,D_F^{-R}(s,s-(\gamma-a_{s,F}(\gamma)))=\frac{1-(\gamma-a_{s,F}(\gamma))/\sqrt{s}} {f_{s,F}^R(\gamma-a_{s,F}(\gamma))} \nonumber \\[3mm]
& & =~\frac{1}{a_{s,F}(\gamma)}\,(1-(\gamma-a_{s,F}(\gamma))/\sqrt{s})\pr a_{\infty}^{-1}(\gamma),~~~~~~ s\pr\infty,
\end{eqnarray}
where it also has been used that $a=a_{s,F}(\gamma)$ satisfies (\ref{5.3}).

\section{Proof of Theorem \ref{thmasss}}\label{proofthm}
We give the proof of \eqref{416}, \eqref{418} in detail; the proof of \eqref{417} and \eqref{419} being quite similar. Take any $\gamma_1,\gamma_2\in\mathbb{R}$ with
\beq\label{81}
\gamma_1<\gamma_{\infty,F}(\eps)<\gamma_2.
\eq
From strict decreasingness of $g_\infty(\gamma)$ and the definition of $\gamma_{\infty, F}(\eps)$, we have
\beq\label{82}
(1+F(1))g_\infty(\gamma_2)<(1+F(1))g_{\infty}(\gamma_{\infty,F}(\eps))=\eps<(1+F(1))g_\infty(\gamma_1).
\eq
Because $g_{s, F}(\gamma)\to (1+F(1))g_\infty(\gamma)$ as $s\to\infty$, we have from \eqref{82} that
\beq\label{84}
g_{s, F}(\gamma_2)<\eps<g_{s, F}(\gamma_1)
\eq
when $s$ is large.
By monotonicity of $g_{s,F}$, this implies that $\gamma_{s,F}(\eps)\in[\gamma_1,\gamma_2]$ for large $s$. Since $\gamma_1$ and $\gamma_2$ in \eqref{81} are arbitrary, it follows that $\gamma_{s,F}(\eps)\to\gamma_{\infty,F}(\eps)$ as $s\to\infty$. In particular, $\gamma_{s,F}(\eps)$ is bounded in $s\geq 1$.

We next show from the weakened form
\beq\label{85}
g_{s, F}(\gamma)=(1+F(1))g_\infty(\gamma)+O(s^{-1/2})
\eq
of \eqref{4999a}
that
\beq\label{86}
\gamma_{s,F}(\eps)=\gamma_{\infty,F}(\eps)+O(s^{-1/2}).
\eq
To that end, we consider \eqref{85} with $\gamma=\gamma_{s,F}(\eps)$ and write
\beq\label{87}
g_\infty(\gamma_{s,F}(\eps))=g_\infty(\gamma_{\infty,F}(\eps))
+(\gamma_{s,F}(\eps)-\gamma_{\infty,F}(\eps))g_\infty'(\gamma_{\infty,F}(\eps))+O((\gamma_{s,F}(\eps)-\gamma_{\infty,F}(\eps))^2).
\eq
Using this, together with
\beq\label{88}
g_{s,F}(\gamma_{s,F}(\eps))=\eps=(1+F(1))g_\infty(\gamma_{\infty,F}(\eps)),
\eq
in \eqref{85} we get
\beq\label{89}
\eps=\eps+(1+F(1))(\gamma_{s,F}(\eps)-\gamma_{\infty,F}(\eps))g_\infty'(\gamma_{\infty,F}(\eps))+O((\gamma_{s,F}(\eps)-\gamma_{\infty,F}(\eps))^2)+O(s^{-1/2}).
\eq
It is known, see \cite[(21)]{ajl}, that $g_\infty'(\gamma)$ is negative and bounded away from $0$ when $\gamma$ is in a bounded set. From \eqref{89} we therefore get \eqref{86}.

We finally show \eqref{416}, and for this we repeat the argument for showing \eqref{86}, but now using the full strength of \eqref{4999a} with $\gamma=\gamma_{s,F}(\eps)$. Using \eqref{86} in \eqref{87} yields
\beq\label{810}
g_\infty(\gamma_{s,F}(\eps))=g_\infty(\gamma_{\infty,F}(\eps))
+(\gamma_{s,F}(\eps)-\gamma_{\infty,F}(\eps))g_\infty'(\gamma_{\infty,F}(\eps))+O(s^{-1}).
\eq
Furthermore, again by \eqref{86},
\beq\label{811}
h_{\infty,F}(\gamma_{s,F}(\eps))=h_{\infty,F}(\gamma_{\infty,F}(\eps))+O(s^{-1/2}).
\eq
When we use this in \eqref{4999a} with $\gamma=\gamma_{s,F}(\eps)$ together with \eqref{88}, we obtain
\beq\label{812}
\eps=\eps+(1+F(1))(\gamma_{s,F}(\eps)-\gamma_{\infty,F}(\eps))g_\infty'(\gamma_{\infty,F}(\eps))+O(s^{-1})+\frac{1}{\sqrt{s}}h_{\infty,F}(\gamma_{\infty,F}(\eps))+O(s^{-1}).
\eq
From this \eqref{416} and \eqref{418} follow at once.

%

\bibliographystyle{plain} 

\appendix

\section{Remaining proofs, except proof of Theorem~\ref{thmm3}}\label{aA}
\mbox{} \\[-9mm]

We start by recalling some basic properties, shown in \cite{ajl}, of the functions $f_s$ and $g_s$ given in (\ref{4.7}). We have
\bi{a.0}
\ITEM{a.} $f_s(\gamma)$ is strictly convex and decreases strictly in $-\infty<\gamma\leq\sqrt{s}$ from $+\infty$ to 0, and $\gamma+f_s(\gamma)$ increases strictly in $-\infty<\gamma\leq\sqrt{s}$ from $0$ to $\sqrt{s}$. Furthermore, $\gamma+f_s(\gamma)=O(\frac{1}{\gamma})$, $\gamma\pr{-}\infty$.
\ITEM{b.} $g_s(\gamma)$ decreases strictly in $-\infty<\gamma\leq\sqrt{s}$ from $\sqrt{s}$ to 0.
\ITEM{c.} $f_s(\gamma)$ and $g_s(\gamma)$ increase in $s\geq1$ to $f_{\infty}(\gamma)=g_{\infty}(\gamma)=\varp(\gamma)/\Phi(\gamma)$ uniformly in any compact set of $\gamma\in\dR$.
\ei
{\bf Proof of Proposition \ref{propp1}.}~We have for $0<\gamma<\sqrt{s}$ that
\beq \label{EC.1}
H_s(\gamma)\leq\sum_{n=0}^{\infty}\,\Bigl(1-\frac{\gamma}{\sqrt{s}}\Bigr)^n=\frac{\sqrt{s}}{\gamma},
\eq
with equality if and only if $p_{s+n}=1$, $n=0,1,...\,$. Hence, $1-\gamma\,H_s(\gamma)/\sqrt{s}\geq0$ for $0<\gamma<\sqrt{s}$, and evidently $1-\gamma\,H_s(\gamma)/\sqrt{s}\geq0$ for $\gamma_P<\gamma\leq0$. Therefore, from (\ref{4.9}), $f_{s,F}^R(\gamma)\geq0$, $\gamma_P<\gamma<\sqrt{s}$, with equality for any $\gamma$ if and only if $p_{s+n}=1$, $n=0,1,...\,$. Next, we write (\ref{4.9}) for $\gamma_P<\gamma<\sqrt{s}$ as
\beq \label{EC.2}
f_{s,F}^R(\gamma)={-}\gamma+\frac{\gamma+f_s(\gamma)}{1+\dfrac{1}{\sqrt{s}}\,f_s(\gamma)\,H_s(\gamma)},
\eq
and then it follows from $\gamma+f_s(\gamma)>0$ that $f_{s,F}^R(\gamma)>{-}\gamma$. This proves the first inequality in {\rm (\ref{4.11})}.

Next, we write (\ref{4.9}) for $\gamma_P<\gamma<\sqrt{s}$ as
\beq \label{EC.3}
f_{s,F}^R(\gamma)=f_s(\gamma)-\frac{1}{\sqrt{s}}\,f_s(\gamma)\,H_s(\gamma)\,\frac{\gamma+f_s(\gamma)} {1+\dfrac{1}{\sqrt{s}}\,f_s(\gamma)\,H_s(\gamma)},
\eq
and it follows from $\gamma+f_s(\gamma)>0$ that $f_{s,F}^R(\gamma)\leq f_s(\gamma)$, with equality if and only if $H_s(\gamma)=0$ if and only if $p_{s+n}=0$, $n=0,1,...\,$. This proves the second inequality in {\rm (\ref{4.11})}.

Next, we write (\ref{4.8}) for $\gamma_P<\gamma<\sqrt{s}$ as
\beq \label{EC.4}
f_{s,F}(\gamma)=f_s(\gamma)+f_s(\gamma)\,H_s(\gamma)\,\frac{1-(\gamma+f_s(\gamma))/\sqrt{s}} {1+f_s(\gamma)\,H_s(\gamma)/\sqrt{s}},
\eq
and it follows from $\gamma+f_s(\gamma)<\sqrt{s}$ that $f_{s,F}(\gamma)\geq f_s(\gamma)$, with equality if and only if $H_s(\gamma)=0$ if and only if $p_{s+n}=0$, $n=0,1,...\,$. This proves the third inequality in {\rm (\ref{4.11})}.

Next, we write (\ref{4.8}) for $\gamma_P<\gamma<\sqrt{s}$ as
\beq \label{EC.5}
f_{s,F}(\gamma)=\sqrt{s}-\gamma-\sqrt{s}\,\frac{\sqrt{s}-\gamma-f_s(\gamma)} {\sqrt{s}+f_s(\gamma)\,H_s(\gamma)},
\eq
and it follows from $\gamma+f_s(\gamma)<\sqrt{s}$ that $f_{s,F}(\gamma)<\sqrt{s}-\gamma$. Since $\sqrt{s}-\gamma>f_s(\gamma)$ for $\gamma<\sqrt{s}$, we have that the function $x\geq0\mapsto(1+x(1-\gamma/\sqrt{s}))/(1+x\,f_s(\gamma)/\sqrt{s})$ is strictly increasing. From (\ref{4.8}) and (\ref{EC.1}) it then follows (with $x=H_s(\gamma)\leq\sqrt{s}/\gamma$) that $f_{s,F}(\gamma)\leq \sqrt{s}\,f_s(\gamma)/(\gamma+f_s(\gamma))$ for $0<\gamma<\sqrt{s}$, with equality if and only if $p_{s+n}=0$, $n=0,1,...\,$. Furthermore, for $\gamma\leq0$, we have
\beq \label{EC.5a}
\sqrt{s}-\gamma-\frac{\sqrt{s}\,f_s(\gamma)}{\gamma+f_s(\gamma)}= \frac{\gamma(\sqrt{s}-\gamma-f_s(\gamma))}{\gamma+f_s(\gamma)}\leq0,
\eq
with equality if and only if $\gamma=0$. This proves the fourth inequality in {\rm (\ref{4.11})}.

The cases of equality in the inequalities in {\rm (\ref{4.11})} have been indicated already along with their proofs, and this settles Proposition~\ref{propp1}, (ii)--(iv).

Proposition~\ref{propp1}(v) follows from 1(i) and the fact that $f_s(\sqrt{s})=0$.

Proposition~\ref{propp1}(vi)--(vii) follow from the representations (\ref{4.9}) and (\ref{4.8}) and the fact that $H_s(\gamma)$ increases to $H_s(\gamma_P+0)$ as $\gamma$ decreases to $\gamma_P$ by non-negativity of all $p_k$. This completes the proof of Proposition~\ref{propp1}. \\ \\
{\bf Note.}~~We have the following consequences of (\ref{EC.3}), (\ref{EC.4}) and $\gamma+f_s(\gamma)=O(\frac{1}{\gamma})$, $\gamma\pr{-}\infty$.
\bi{a.0}
\ITEM{a.} $f_{s,F}^R(\gamma)=f_s(\gamma)+O(\frac{1}{\gamma})$, $\gamma_P<\gamma<0$,
\ITEM{b.} $f_{s,F}(\gamma)=f_s(\gamma)+\sqrt{s}+O(\frac{1}{\gamma})$, $\gamma_P<\gamma<0$, when $F\not\equiv0$,
\ITEM{c.} $g_{s,F}^R(\gamma),g_{s,F}(\gamma)=\sqrt{s}+O(\frac{1}{\gamma})$, $\gamma_P<\gamma<0$.
\ei
These results are in particular relevant when $P=0$ so that $\gamma_P={-}\infty$. \\ \\
{\bf Proof of Proposition \ref{propp2}.}
{(i)} Since $f_{s,F}(\gamma)=(1-\gamma/\sqrt{s})\,g_{s,F}(\gamma)$, it is sufficient to show that $g_{s,F}(\gamma)$ is strictly decreasing. We have from (\ref{EC.5}) that
\beq \label{EC.6}
g_{s,F}(\gamma)=\frac{f_{s,F}(\gamma)} {1-\gamma/\sqrt{s}}=\sqrt{s}\Bigl(1-\frac{1-\dfrac{1}{\sqrt{s}}\,g_s(\gamma)} {1+\dfrac{1}{\sqrt{s}}\,f_s(\gamma)\,H_s(\gamma)}\Bigr).
\eq
Now $g_s(\gamma)$ strictly decreases in $-\infty<\gamma\leq\sqrt{s}$ from $\sqrt{s}$ to 0, and $f_s(\gamma)\,H_s(\gamma)$ is non-negative and decreasing in $\gamma_P<\gamma\leq\sqrt{s}$. It follows that
\beq \label{EC.7}
\frac{1-\dfrac{1}{\sqrt{s}}\,g_s(\gamma)} {1+\dfrac{1}{\sqrt{s}}\,f_s(\gamma)\,H_s(\gamma)}
\eq
is non-negative and strictly increasing in $\gamma_P<\gamma\leq\sqrt{s}$, and the proof is complete.\\
{(ii)} We have that $\gamma+f_s(\gamma)$ is positive and strictly increasing in $-\infty<\gamma\leq\sqrt{s}$, and $f_s(\gamma)\,H_s(\gamma)$ is non-negative and decreasing in $\gamma_P<\gamma\leq\sqrt{s}$. It follows that
\beq \label{EC.8}
\frac{\gamma+f_s(\gamma)} {1+\dfrac{1}{\sqrt{s}}\,f_s(\gamma)\,H_s(\gamma)}
\eq
is strictly increasing in $\gamma_P<\gamma\leq\sqrt{s}$. Then it follows from the representation (\ref{EC.2}) of $f_{s,F}^R$ that $\gamma+f_{s,F}^R(\gamma)$ is strictly increasing in $\gamma_P<\gamma\leq\sqrt{s}$.
\mbox{} \\ \\
{\bf Proof of Proposition \ref{propp3}.} With $p_k=p\in(0,1)$ for $k\geq s$, we have $f_{s,F}$ and $f_{s,F}^R$ as in (\ref{4.16}) so that, in particular, $f_{s,F}^R=(1-p)\,f_{s,F}$.

{(i)} We have that $\gamma+f_s(\gamma)$ is positive and strictly increasing in $-\infty<\gamma\leq\sqrt{s}$ and that $f_s(\gamma)$ is positive and strictly decreasing in $-\infty<\gamma\leq\sqrt{s}$. Hence, from (\ref{4.16}), both $f_{s,F}(\gamma)$ and $f_{s,F}^R(\gamma)$ are strictly decreasing in $\gamma_P<\gamma\leq\sqrt{s}$. We next show (strict) convexity of $f_{s,F}^R$. We write (\ref{4.16}) for $\gamma_P<\gamma<\sqrt{s}$ as
\beq \label{EC.9}
f_{s,F}^R(\gamma)=\frac{f_s(\gamma)}{1+b(\gamma+f_s(\gamma))}~;~~~~~~b=\frac{p}{(1-p)\,\sqrt{s}}=\frac{-1}{\gamma_P} >0,
\eq
and we compute for $\gamma_P<\gamma\leq\sqrt{s}$
\begin{eqnarray} \label{EC.10}
& \mbox{} & \hspace*{-1cm}(f_{s,F}^R)''(\gamma)=\Bigl(\frac{f_s'(\gamma)(1+b\gamma)-bf_s(\gamma)} {(1+b(\gamma+f_s(\gamma)))^2} \Bigr)'~\nonumber \\[3.5mm]
& & \hspace*{-1cm}=~\frac{f_s''(\gamma)(1{+}b\gamma)(1{+}b(\gamma{+}f_s(\gamma))){-}2b(f_s'(\gamma)(1{+}b\gamma)
{-}bf_s(\gamma))(1{+}f_s'(\gamma))} {(1+b(\gamma+f_s(\gamma)))^3}.
\end{eqnarray}
From
\beq \label{EC.11}
f_s''(\gamma)\,,~~1+b\gamma>0\,,~~f_s(\gamma>\max\{0,{-}\gamma\}\,,~~-1<f_s'(\gamma)<0
\eq
for $\gamma_P<\gamma<\sqrt{s}$, it follows that $(f_{s,F}^R)''(\gamma)>0$ for $\gamma_P<\gamma<\sqrt{s}$. Hence, $f_{s,F}^R(\gamma)$ is strictly convex in $\gamma_P<\gamma<\sqrt{s}$, and so is $f_{s,F}(\gamma)$, see \eqref{4.16}.

The values of $f_{s,F}$, $f_{s,F}^R$ at $\gamma=\gamma_P+0$ and $\gamma=\sqrt{s}$, as given in and below (\ref{4.17}), follow from Proposition 1(v)--(vii) and the fact that $H_s(\gamma_P+0)=\infty$, see (\ref{4.15}), in this case.

{(ii)} We have for $\gamma_P<\gamma<\sqrt{s}$ from (\ref{4.3}), (\ref{4.4}), (\ref{4.7}) and (\ref{4.16}) that
\beq \label{EC.12}
g_{s,F}(\gamma)=\frac{g_s(\gamma)} {1-p+\dfrac{1}{\sqrt{s}}\,p(\gamma+f_s(\gamma))}.
\eq
Now $g_s(\gamma)$ is strictly decreasing and positive in $-\infty<\gamma<\sqrt{s}$ and $\gamma+f_s(\gamma)$ is strictly increasing and positive in $-\infty<\gamma<\sqrt{s}$. Hence, $g_{s,F}(\gamma)$ is positive and strictly decreasing in $\gamma_P<\gamma<\sqrt{s}$, and so is $g_{s,F}^R(\gamma)=(1-p)\,g_{s,F}(\gamma)$. The values assumed by $g_{s,F}$ and $g_{s,F}^R$ at $\gamma=\gamma_P+0$ and $\gamma=\sqrt{s}$ follow from (i).

{(iii)} We have from (i) that $\gamma+f_{s,F}^R(\gamma)$ is strictly convex in $\gamma_P<\gamma<\sqrt{s}$, and $(f_{s,F}^R)'(\gamma)>{-}1$ for $\gamma_P<\gamma\leq\sqrt{s}$ by Proposition~\ref{propp2}(ii). The values assumed by $\gamma+f_{s,F}^R(\gamma)$ at $\gamma=\gamma_P+0$ and $\gamma=\sqrt{s}$ follow from (i).

{(iv)} Strict convexity of $\gamma+f_{s,F}(\gamma)$ in $\gamma_P<\gamma<\sqrt{s}$ follows from (i). The values assumed by $\gamma+f_{s,F}(\gamma$ at $\gamma=\gamma_P+0$ and $\gamma=\sqrt{s}$ also follow from (i). From these values it is seen that $\gamma+f_{s,F}(\gamma)$ is non-monotonic in $\gamma_P<\gamma<\sqrt{s}$.

This completes the proof of Proposition \ref{propp3}.

\mbox{} \\
{\bf Note.}~~It follows from the expression for $(f_{s,F}^R)'(\gamma)$ in (\ref{EC.10}) that
\beq \label{EC.13}
(f_{s,F}^R)'(\gamma_P+0)=\frac{\gamma_P}{f_s(\gamma_P)}>{-}1,
\eq
and that
\beq \label{EC.14}
(f_{s,F})'(\gamma_P+0)=\frac{1}{1-p}~\frac{\gamma_P}{f_s(\gamma_P)}<{-}1,
\eq
where the two inequalities follow from $0<\delta+f_s(\delta)<\sqrt{s}$ and the definition of $\gamma_P$, see (\ref{EC.9}). \\

\noindent {\bf Proof of Proposition \ref{propp6}.}
With $p_k=1$, $s\leq k\leq N$, $p_k=0$, $k>N$, we have $P=0$ and $H_s(\gamma)=\gamma^{-1}\sqrt{s}(1-(1-\gamma/\sqrt{s})^{N+1})$. For $\gamma\leq \sqrt{s}$ it then follows from (\ref{4.9}) and (\ref{4.4}) that
\beq \label{EC.15}
g_{s,F}^R(\gamma)=\frac{(1-\gamma/\sqrt{s})^N\,f_s(\gamma)} {1+\dfrac{1}{\sqrt{s}}\,f_s(\gamma)\,\dsum_{n=0}^N\, (1-\gamma/\sqrt{s})^n}.
\eq
A computation shows that
\begin{eqnarray} \label{EC.16}
& \mbox{} & \sqrt{s}(g_{s,F}^R)'(\gamma)\,\frac{(\sqrt{s}+f_s(\gamma)\,H_s(\gamma))^2} {(1-\gamma/\sqrt{s})^{N-1}}~ \nonumber \\[3mm]
& & =~{-}N\,f_s(\gamma)+\Bigl(1-\frac{\gamma}{\sqrt{s}}\Bigr)\sqrt{s}\,f_s'(\gamma)-\frac{1}{\sqrt{s}}\, f_s^2(\gamma) \,\sum_{n=0}^N\,(N-n)\Bigl(1-\frac{\gamma}{\sqrt{s}}\Bigr)^n,
\end{eqnarray}
and all terms on the second line of (\ref{EC.16}) are negative. Hence, $g_{s,F}^R$ is strictly decreasing, as required. \\

\noindent{\bf Some examples.}~~Take $H_s(\gamma)$ in \eqref{4.10} of the form
\beq \label{EC.17}
H_s(\gamma)=\sum_{n=0}^\infty\:q_n\,\Bigl(1-\frac{\gamma}{\sqrt{s}}\Bigr)^n, \quad  q_n=\frac{\eps\,P^n}{(n+1)^{\alpha}},
\eq
with $0<\eps<1$, $\alpha>0$, $0<P<1$.
Then $0<q_n<1$, $q_n$ is decreasing in $n$ and ${{\rm lim}\ {\rm sup}}\ q_n^{1/(n+1)}=P$. When $\alpha>1$, we have that $H_s(\gamma_P+0)<\infty$. Consequently, in this case
\beq
1-P<D_F^R(s,\lambda_P-0)<D_F(s,\lambda_P-0)<1.
\eq
When $1<\alpha<2$ we have that $H_s'(\gamma_P+0)=-\infty$, and in this case
$(D_F^R)'(s,\lambda_P-0)=-\infty$, showing that $D_F^R(s,\lambda)$ is increasing for $\lambda$ close to $\lambda_P$. An example where $F(\frac1P-0)=H_s(\gamma_P+0)=\infty$ and $g_{s,F}^R(\gamma_P+0)=\infty=-(D_F^R)'(s,\lambda_P-0)$ is provided by the choice $s=1$ and $q_n=\frac{1}{10}{-1/2 \choose n}(-\frac12)^n$, $n=0,1,\dots$., yielding $H_1(\gamma)=\frac{1}{10}(1+\gamma)^{-1/2}$ with $P=2$, $\gamma_P=-1$. \\

\noindent{\bf Proof of Lemma \ref{lemdec}.}~~We have from (\ref{4.7}) and (\ref{4.8}) that
\begin{eqnarray} \label{EC.19}
& \mbox{} & \frac{1}{\sqrt{s}}\,D_F^{-1}(s,s-\gamma\sqrt{s})=\frac{1-\gamma/\sqrt{s}}{f_{s,F}(\gamma)}~ \nonumber \\[3.5mm]
& & =~\frac{1-\gamma/\sqrt{s}}{f_s(\gamma)}~\frac{1+\dfrac{1}{\sqrt{s}}\,f_s(\gamma)\,H_s(\gamma)} {1+(1-\gamma/\sqrt{s})\,H_s(\gamma)}~ \nonumber \\[3.5mm]
& & =~\frac{1-\gamma/\sqrt{s}}{f_s(\gamma)}~\dfrac{1-\dfrac{1}{\sqrt{s}}\,\gamma\,H_s(\gamma)+\dfrac{1}{\sqrt{s}} \,H_s(\gamma)(f_s(\gamma)+\gamma)}{1+(1-\gamma/\sqrt{s})\,H_s(\gamma)}~ \nonumber \\[3.5mm]
& & =~(1-p(\gamma))\,\frac{1-\gamma/\sqrt{s}}{f_s(\gamma)}+p(\gamma)\,\frac{1}{\sqrt{s}}~\frac{1-\gamma/\sqrt{s}} {f_s(\gamma)}\,(f_s(\gamma)+\gamma),
\end{eqnarray}
with $p(\gamma)=q_\lambda$ where $q_\lambda$ is given in (\ref{4.31}). Now, by (\ref{4.7}),
\beq \label{EC.20}
\frac{1-\gamma/\sqrt{s}}{f_s(\gamma)}=\frac{1}{\sqrt{s}}\,B^{-1}(s,s-\gamma\sqrt{s}),
\eq
and
\begin{align} \label{EC.21}
\frac{1-\gamma/\sqrt{s}}{f_s(\gamma)}\,(f_s(\gamma)+\gamma) & =  1-\frac{\gamma}{\sqrt{s}}+\frac{\gamma}{\sqrt{s}}\, B^{-1}(s,s-\gamma\sqrt{s})
 =  C^{-1}(s,s-\gamma\sqrt{s}).
\end{align}
Then (\ref{4.29}) follows from (\ref{EC.19})--(\ref{EC.21}). The proof of (\ref{4.30}) is similar.

We observe that $0\leq p(\gamma)\leq1$, which follows from (\ref{EC.1}). \\

\noindent{\bf Proof of Theorem~\ref{thmm1}.}
From \cite[Section 3]{fa} one can extract a proof of Theorem \ref{thmm1}. This proof depends on basic properties of birth-death processes. The proof that we give here is presented in QED coordinates and uses the analytic properties of $f_{s,F}^R$, such as given in Propositions \ref{propp1} and \ref{propp2}. The proof is given under general conditions, so that also $P=1$ and $F(\frac1P-0),H_s(\gamma_P+0)<\infty$ is allowed.

We distinguish the following cases:

\noindent{\bf Case a.}~~$P=0$. Then $\gamma_P={-}\infty$, and both $H_s(\gamma)$ and $f_{s,F}^R(\gamma)$ are well-defined and analytic in $-\infty<\gamma\leq\sqrt{s}$. Set $\gamma_{P,F}=0$.

\noindent{\bf Case b.}~~$P\in(0,1]$. Then $\gamma_P={-}(1-P)\sqrt{s}/P\in({-}\infty,0]$, and both $H_s(\gamma)$ and $f_{s,F}^R(\gamma)$ are well-defined and analytic in $\gamma_P<\gamma\leq\sqrt{s}$. In this case~b, we distinguish the subcases
\bi{b2.0}
\ITEM{b1.} $H_s(\gamma_P+0)<\infty$. Then by Abel's theorem, $H_s(\gamma)$ is continuous in $\gamma_P\leq\gamma\leq\sqrt{s}$, and so is $f_{s,F}^R(\gamma)$.
\ITEM{b2.} $H_s(\gamma_P+0)=\infty$.
\ei
In these cases we have from (\ref{4.9}) and (\ref{EC.2})
\beq \label{5.4}
f_{s,F}^R(\gamma_P+0)={-}\gamma_P+\frac{\gamma_P+f_s(\gamma_P)}{1+\dfrac{1}{\sqrt{s}}\,f_s(\gamma_P)\, H_s(\gamma_P+0)}~.
\eq
Hence $f_{s,F}^R(\gamma_P+0)>{-}\gamma_P$ in case~b1 while $f_{s,F}^R(\gamma_P+0)={-}\gamma_P$ in case b2. Set $\gamma_{P,F}=f_{s,F}^R(\gamma_P+0)+\gamma_P$. \\

In case $P=0$, we have by Proposition~\ref{propp1} and \ref{propp2} that
\beq \label{EC.22}
f_{s,F}^R(\gamma)>{-}\gamma,~~~~~~(f_{s,F}^R)'(\gamma)>{-}1,
\eq
for $\gamma<\sqrt{s}$, and it follows as in the proof of \cite{ajl}, Theorem~\ref{thmm3} in Section~4.3 that for any $\gamma\in(0,\sqrt{s})$ there is a unique solution $a$ of the equation $a=f_{s,F}^R(\gamma-a)$. This solution, $a_{s,F}(\gamma)$, satisfies $a_{s,F}(\gamma)\pr{+}\infty$ as $\gamma\downarrow0$. For if $b:=\lim{\rm inf}_{\gamma\downarrow0}\,a_{s,F}(\gamma)<\infty$, we would have $b=\lim{\rm inf}_{\gamma\downarrow0}\, f(\gamma-a_{s,F}(\gamma))=f({-}b)$, contradicting the first item in (\ref{EC.22}). Similarly, it can be shown, compare the beginning of the proof of \cite{ajl}, Theorem~8 in Section~4.8, by considering $c:=\lim{\rm sup}_{\gamma\uparrow\sqrt{s}}\,a_{s,F}(\gamma)$, that $a_{s,F}(\gamma)\pr0$ as $\gamma\uparrow\sqrt{s}$.
Assume now that $P\in(0,1]$, in which we exclude the case that $p_k=1$, $k\geq s$. Now (\ref{EC.22}) holds for $\gamma_P<\gamma<\sqrt{s}$. Let $\gamma\in(\gamma_{P,F},\sqrt{s})$. We have
\beq \label{EC.23}
f_{s,F}^R(\gamma-0)>0,
\eq
while
\beq \label{EC.24}
f_{s,F}^R(\gamma-(\gamma-\gamma_P)+0)=\gamma_{P,F}-\gamma_P<\gamma-\gamma_P,
\eq
and so
\beq \label{EC.25}
f_{s,F}^R(\gamma-\delta)<\delta
\eq
when $\delta$ is less than but close to $\gamma-\gamma_P>0$. By continuity, it follows from (\ref{EC.23}) and (\ref{EC.25}) that the equation $a=f_{s,F}^R(\gamma-a)$ has a solution $a$, and this solution is unique by (\ref{EC.22}). To show that this solution, $a_{s,F}(\gamma)$, satisfies $a_{s,F}(\gamma)\pr f_{s,F}^R(\gamma_P+0)$ as $\gamma\downarrow\gamma_{P,F}$, we need the following lemma.
\mbox{}
\begin{lem}~~Let $c<0<d$ and let $h:(c,d)\pr\dR$ be smooth and such that $h(a)>0$, $h'(a)<1$ for $c<a<d$ while $h(a)\pr d$ when $a\uparrow d$. Then for any $\eps$, $0<\eps<{-}c$, there is a unique solution $a(\eps)$ of the equation $a=h(a-\eps)$, and $a(\eps)\pr d$ as $\eps\downarrow0$.
\end{lem}
\noindent{\bf Proof.}~~Let $\eps\in(0,{-}c)$. Now
\beq \label{EC.26}
h(0-\eps)-0=h({-}\eps)>0~;~~~~~~h(a-\eps)-a\pr{-}\eps<0\,,~~a\uparrow d+\eps.
\eq
Hence, by continuity, there is an $a\in(0,d+\eps)$ such that $a=h(a-\eps)$, and this $a$ is unique by the assumption that $h'(b)<1$ for $b\in(c,d)$. To proceed, we first show that $h(a)>a$ when $c<a<d$. Indeed, when $h(a)<a$ for some $a\in(c,d)$, we would have
\beq \label{EC.27}
h(a_1)-h(a)=h(a_1)-a_1-(h(a)-a)+a-a_1>a_1-a
\eq
when $a_1$ is sufficiently close to $d$, contradicting $h'(b)<1$ for all $b\in(c,d)$. Therefore, $h(a)\geq a$ for $a\in(c,d)$, and $h(a)=a$ for some $a\in(c,d)$ cannot occur either. For otherwise, we would have for $a_1>a$ that
\beq \label {EC.28}
h(a_1)-h(a)\geq a_1-a,
\eq
contradicting $h'(b)<1$ for all $b\in(c,d)$.

We now show that $a(\eps)\pr d$ as $\eps\downarrow0$. We have from $a(\eps)<d+\eps$, see (\ref{EC.26}), that $\lim{\rm sup}_{\eps\downarrow0}\,a(\eps)\leq d$. Now suppose that $b:=\lim{\rm inf}_{\eps\downarrow0}\,a(\eps)<d$. Take $\eps_n>0$, $\eps_n\pr0$ such that $a(\eps_n)\pr b<d$. Then
\beq \label{EC.29}
0=h(a(\eps_n)-\eps_n)-a(\eps_n)\pr h(b)-b,
\eq
contradicting $h(b)>b$ for all $b\in(c,d)$. This completes the proof of the lemma. \\
\mbox{}

Taking in the lemma
\beq \label{EC.30}
h(a)=f_{s,F}^R(\gamma_{P,F}-a),~~~~~~c:={-}\sqrt{s}+\gamma_{P,F}<a<f_{s,F}^R(\gamma_P+0)=:d,
\eq
it follows that $a_{s,F}(\gamma)\pr f_{s,F}^R(\gamma_P+0)$ as $\gamma\downarrow\gamma_{P,F}$. Finally, $a_{s,F}(\gamma)\pr0$ as $\gamma\uparrow\sqrt{s}$ can be shown by using the same argument as in \cite{ajl}, Theorem~8 in Section~4.8 for showing that $a(\gamma)\pr0$ as $\gamma\uparrow\sqrt{s}$. This completes the proof.

\section{Proof of Theorem~\ref{thmm3}}\label{aB}
\mbox{} \\[-9mm]

In this appendix we present the proof of Theorem~\ref{thmm3} on the function $L_s$ in (\ref{7.10}), given in terms of the function $a_s(\gamma)$ that solves Cohen's equation $a=f_s(\gamma-a)$. The proofs rely heavily on (extensions of) the results in \cite{ajl}. In particular, we use
\beq \label{EC.31}
\gamma\,a_s(\gamma)=1-\frac{2\gamma}{\sqrt{s}}-\Bigl(1-\frac2s\Bigr)\gamma^2+4\Bigl(1-\frac1s\Bigr)\,\frac{\gamma^3} {\sqrt{s}}+O(\gamma^4),~~~~~\gamma\downarrow0,
\eq
which is a sharpening of \cite[Theorem~3]{ajl}. This sharpening can be obtained by the method to prove \cite[Theorem~3]{ajl} where, as an intermediate step, \cite[Proposition~2]{ajl} should be sharpened to
\beq \label{EC.32}
f_s(\delta)={-}\delta-\frac{1}{\delta}-\frac{2}{\delta^2\sqrt{s}}+\Bigl(2-\frac6s\Bigr)\,\frac{1}{\delta^3}+ \Bigl(16-\frac{24}{s}\Bigr)\,\frac{1}{\delta^4\sqrt{s}}+O\Bigl(\frac{1}{\delta^5}\Bigr),~~~~~\delta\pr{-}\infty,
\eq
using the methods of \cite[Section~4.1]{ajl}. We shall also use and sharpen \cite[Proposition~1]{ajl},
\beq \label{EC.33}
1-\frac{2}{\sqrt{s}}\,\gamma-\gamma^2<\gamma\,a_s(\gamma)<1-\frac{1}{\sqrt{s}}\,\gamma,~~~~~~0<\gamma<\sqrt{s}.
\eq
\mbox{} \\
{\bf Proof of $0<L_s(\gamma)<\sqrt{s}-\gamma$, $0<\gamma<\sqrt{s}$.}~~This follows from the definition in (\ref{7.10}) and $a_s(\gamma)>0$. \\ \\
{\bf Proof of $L_s(\gamma)=\gamma s(1+O(\gamma\sqrt{s}))$, $\gamma\downarrow0$.}~~This follows from the definition in (\ref{7.10}) and (\ref{EC.33}). \\ \\
{\bf Proof of $L_s(\gamma)=(\sqrt{s}-\gamma)(1+O(s^{-1/2}\,e^s(1-\gamma/\sqrt{s})^s))$, $\gamma\uparrow\sqrt{s}$.}~~This follows from the proof of \cite[Theorem~4]{ajl}, and Stirling's formula. \\
\mbox{}

We next show the results on unimodality.

\begin{prop}\label{propEC21}
We have for $0<\gamma<\sqrt{s}$
\beq \label{EC.34}
L_s'(\gamma)=0\Leftrightarrow \gamma\,a_s(\gamma)=\tfrac12\,(1-\gamma/\sqrt{s}).
\eq
\end{prop}

\noindent {\bf Proof.}
We have from the definition of $L_s$ in (\ref{7.10})
\beq \label{EC.35}
L_s'(\gamma)=0\Leftrightarrow\sqrt{s}-\gamma+2a_s(\gamma)+(\sqrt{s}-\gamma)\,a_s'(\gamma)=0.
\eq
By implicit differentiation in  \cite[(14)]{ajl} and the expression in \cite[Subsection~4.3]{ajl} for $f_s'$ in terms of $f_s$ we have
\beq \label{EC.36}
a_s'(\gamma)=\frac{-a_s(\gamma)(\gamma+1/\sqrt{s})}{1-\gamma/\sqrt{s}-\gamma\,a_s(\gamma)},~~~~~~0<\gamma<\sqrt{s}.
\eq
Using this in (\ref{EC.35}) with the facts that $\gamma>0$ and $1-\gamma/\sqrt{s}-\gamma\,a_s(\gamma)>0$, we have for $0<\gamma<\sqrt{s}$
\beq \label{EC.37}
L_s'(\gamma)=0\Leftrightarrow(\gamma\,a_s(\gamma))^2+(1-\gamma/\sqrt{s})(\gamma\sqrt{s}-\tfrac12)\,\gamma\,a_s(\gamma)-
\tfrac12\,(1-\gamma/\sqrt{s})^2\,\gamma\sqrt{s}=0.
\eq
The quadratic in $\gamma\,a_s(\gamma)$ occurring in the second proposition in (\ref{EC.37}) has the roots
\beq \label{EC.38}
\gamma\,a_s(\gamma)={-}\tfrac12\,(1-\gamma/\sqrt{s})(\gamma\sqrt{s}-\tfrac12)\pm\tfrac12\,
(1-\gamma/\sqrt{s})(\gamma\sqrt{s}+\tfrac12).
\eq
Since $\gamma\,a_s(\gamma)>0$, only the root in (\ref{EC.38}) with the $+$-sign needs to be considered. The latter root equals $\frac12\,(1-\gamma/\sqrt{s})$, and this completes the proof.\\

To show unimodality of $L_s$, we should consider the function $\gamma\,a_s(\gamma)/(1-\gamma/\sqrt{s})$, $0<\gamma<\sqrt{s}$. This function assumes the values 1 and 0 at $\gamma=0{+}$ and $\gamma=\sqrt{s}-0$, and so, by Proposition~\ref{propEC21}, it is sufficient to show that this function is strictly decreasing in $0<\gamma<\sqrt{s}$. The result we show below is somewhat stronger and will also be used in the proof of Proposition~\ref{propEC23}.

\begin{prop}\label{propEC22}
$\gamma\,a_s(\gamma)/(1-\gamma/\sqrt{s})^2$ decreases strictly in $0<\gamma<\sqrt{s}$ when $s>1$.
\end{prop}
\noindent {\bf Proof.}
 We compute
\begin{eqnarray} \label{EC.39} ((1-\gamma/\sqrt{s})^{-2}\,\gamma\,a_s(\gamma))'=(1-\gamma/\sqrt{s})^{-2}\,\Bigl[\frac{2}{\sqrt{s}}\,(1-\gamma/\sqrt{s})^{-1}\,\gamma\,a_s(\gamma)+a_s(\gamma) -\gamma\,a_s'(\gamma)\Bigr].
\end{eqnarray}
Using (\ref{EC.36}), we thus see that for $0<\gamma<\sqrt{s}$
\begin{eqnarray} \label{EC.40}
& \mbox{} & \hspace*{-8mm}((1-\gamma/\sqrt{s})^{-2}\,\gamma\,a_s(\gamma))'<0~ \nonumber \\[3.5mm]
& & \hspace*{-8mm}\Leftrightarrow~(1-\gamma\,a_s(\gamma)-\gamma/\sqrt{s})\Bigl(\frac{2\gamma}{\sqrt{s}}+1-
\gamma/\sqrt{s}\Bigr) -\gamma(\gamma+1/\sqrt{s})(1-\gamma/\sqrt{s})<0~ \nonumber \\[3.5mm]
& & \hspace*{-8mm}\Leftrightarrow~(1-\gamma/\sqrt{s})(1-\gamma^2)-\gamma\,a_s(\gamma)(1+\gamma/\sqrt{s})<0.
\end{eqnarray}
It is therefore sufficient to show that for $0<\gamma<\sqrt{s}$
\beq \label{EC.41}
\gamma\,a_s(\gamma)>(1-\gamma^2)\,\frac{1-\gamma/\sqrt{s}}{1+\gamma/\sqrt{s}}.
\eq

We compute
\beq \label{EC.42}
\frac{1-\gamma/\sqrt{s}}{1+\gamma/\sqrt{s}}\,(1-\gamma^2)=1-\frac{2\gamma}{\sqrt{s}}-\Bigl(1-\frac2s\Bigr)\,\gamma^2 +2\Bigl(1-\frac1s\Bigr)\,\frac{\gamma^3}{\sqrt{s}}+O(\gamma^4),~~~~~\gamma\downarrow0,
\eq
and so, by (\ref{EC.31}), we see that (\ref{EC.41}) holds for small positive $\gamma$. Now suppose that $\gamma$, $0<\gamma<\sqrt{s}$, is such that
\beq \label{EC.43}
\gamma\,a_s(\gamma)=(1-\gamma^2)\,\frac{1-\gamma/\sqrt{s}}{1+\gamma/\sqrt{s}}.
\eq
At such a $\gamma$ we compute, using (\ref{EC.36}) and (\ref{EC.42}) twice,
\begin{eqnarray} \label{EC.44}
\hspace*{-5mm}(\gamma\,a_s(\gamma))' & = & a_s(\gamma)-\frac{\gamma\,a_s(\gamma)(\gamma+1/\sqrt{s})}{1-\gamma/\sqrt{s}-\gamma\,a_s(\gamma)}~ \nonumber \\[3.5mm]
& = & a_s(\gamma)-\frac{1+\gamma/\sqrt{s}}{1-\gamma/\sqrt{s}}\,a_s(\gamma)={-}\frac{2}{\sqrt{s}}~\frac{1-\gamma^2} {1+\gamma/\sqrt{s}}.
\end{eqnarray}
At the same time, we compute
\beq \label{EC.45}
\Bigl((1-\gamma^2)\,\frac{1-\gamma/\sqrt{s}}{1+\gamma/\sqrt{s}}\Bigr)'={-}\frac{2}{\sqrt{s}}~ \frac{1+\gamma\sqrt{s}-\gamma^2-\gamma^3/\sqrt{s}}{(1+\gamma/\sqrt{s})^2}.
\eq
Since $s>1$, we have for $0<\gamma<\sqrt{s}$
\beq \label{EC.46}
1+\gamma\sqrt{s}-\gamma^2-\gamma^3/\sqrt{s}>(1+\gamma/\sqrt{s})(1-\gamma^2)=1+\gamma/\sqrt{s}-\gamma^2-
\gamma^3/\sqrt{s}.
\eq
Hence, at a $\gamma\in(0,\sqrt{s})$ where (\ref{EC.43}) holds, we have
\beq \label{EC.47}
(\gamma\,a_s(\gamma))'>\Bigl((1-\gamma^2)\,\frac{1-\gamma/\sqrt{s}}{1+\gamma/\sqrt{s}}\Bigr)'.
\eq
This is in particular so for
\beq \label{EC.48}
\gamma_0:={\rm inf}\,\{0<\gamma<\sqrt{s}\:|\:\mbox{(\ref{EC.43}) holds}\}.
\eq
This $\gamma_0\in(0,\sqrt{s})$ since (\ref{EC.41}) holds for small positive $\gamma$ and since we have assumed that there is a $\gamma\in(0,\sqrt{s})$ such that (\ref{EC.43}) holds. However, validity of (\ref{EC.43}) and (\ref{EC.47}) for $\gamma=\gamma_0$ implies that
\beq \label{EC.49}
\gamma\,a_s(\gamma)<(1-\gamma^2)\,\frac{1-\gamma/\sqrt{s}}{1+\gamma/\sqrt{s}}
\eq
holds for $\gamma$'s close to but less than $\gamma_0$. However, (\ref{EC.41}) holds for $\gamma$'s close to 0, and so there is, by continuity, a $\gamma_1<\gamma_0$ such that (\ref{EC.43}) holds. Contradiction, see (\ref{EC.48}). This proves that (\ref{EC.41}) holds for all $\gamma\in(0,\sqrt{s})$, and the proof is complete.

\noindent{\bf Note.}~~We have $\gamma\,a_{s=1}(\gamma)=(1-\gamma)^2$, $0<\gamma\leq1$.

We have from Proposition~\ref{propEC22} that there is a unique root $\gamma=\hat{\gamma}_s\in(0,\sqrt{s})$ of
\beq \label{EC.50}
\gamma\,a_s(\gamma)=\tfrac12\,(1-\gamma/\sqrt{s}).
\eq
From the behavior of $L_s(\gamma)$ near $\gamma=0$ and $\gamma=\sqrt{s}$, it thus follows that $L_s$ is unimodal, with unique maximum at $\gamma=\hat{\gamma}_s$. The value of $L_s(\gamma)$ at $\gamma=\hat{\gamma}_s$, see (\ref{7.15}), is easily obtained by inserting (\ref{EC.50}) into the definition of $L_s$ in (\ref{7.10}).

We proceed by showing some properties of $\hat{\gamma}_s$.
\begin{prop}\label{propEC23}For $s>1$,
\beq \label{EC.51}
\hat{\gamma}_s>f_s(0)>\Bigl(\frac12+\frac{1}{16s}\Bigr)^{1/2}-\frac{1}{4\sqrt{s}}.
\eq
\end{prop}
\begin{proof}
It follows from $a_s(\gamma)=f_s(\gamma-a_s(\gamma))$ that $a_s(f_s(0))=f_s(0)$. Using the inequality in (\ref{EC.41}) with $\gamma=f_s(0)$, we get
\beq \label{EC.52}
f_s^2(0)=(1-f_s^2(0))\,\frac{1-f_s(0)/\sqrt{s}}{1+f_s(0)/\sqrt{s}}.
\eq
Working this out (noting that the cubic terms $f_s^3(0)$ cancel), we find
\beq \label{EC.53}
f_s^2(0)>\frac12\,\Bigl(1-\frac{f_s(0)}{\sqrt{s}}\Bigr).
\eq
That is, the value assumed by $\gamma\,a_s(\gamma)/(1-\gamma/\sqrt{s})$ at $\gamma=f_s(0)$ exceeds $1/2$. Since $\gamma\,a_s(\gamma)/(1-\gamma/\sqrt{s})$ is strictly decreasing in $0<\gamma<\sqrt{s}$, the first inequality in (\ref{EC.51}) follows. Finally, from (\ref{EC.53}) we also get
\beq \label{EC.54}
\Bigl(f_s(0)+\frac{1}{4\sqrt{s}}\Bigr)^2-\frac12-\frac{1}{16s}>0,
\eq
and this gives the second inequality in (\ref{EC.51}).
\end{proof}

\noindent{\bf Note.}~~We have $\hat{\gamma}_{s=1}=1/2$.
\begin{thm}\label{thmEC21}$\hat{\gamma}_s$ increases strictly in $s\geq1$ from $1/2$ at $s=1$ to $1.034113461...$ at $s=\infty$.
\end{thm}
\begin{proof}
Let $s>t\geq1$. From (\ref{EC.33}) we have for $u\geq1$ that
\beq \label{EC.55}
p_u(\gamma):=\frac{\gamma\,a_u(\gamma)}{1-\gamma/\sqrt{u}}=1-\gamma/\sqrt{u}+O(\gamma^2),~~~~~~\gamma\downarrow0,
\eq
and so
\beq \label{EC.56}
\frac{\gamma\,a_s(\gamma)}{1-\gamma/\sqrt{s}}>\frac{\gamma\,a_t(\gamma)}{1-\gamma/\sqrt{t}}
\eq
holds for small positive $\gamma$. Now suppose we have a $\gamma\in(0,\sqrt{t})$ such that
\beq \label{EC.57}
p:=\frac{\gamma\,a_s(\gamma)}{1-\gamma/\sqrt{s}}=\frac{\gamma\,a_t(\gamma)}{1-\gamma/\sqrt{t}}.
\eq
We have $p\in(0,1)$. We compute, using (\ref{EC.36}) and the definition of $p_u$ in (\ref{EC.55})
\begin{align} \label{EC.58}
\hspace*{-1cm}\Bigl(\frac{a_u(\gamma)}{1-\gamma/\sqrt{u}}\Bigr)'&=\frac{-a_u(\gamma)}{1-\gamma/\sqrt{u}}~\frac{\gamma+1\sqrt{u}}{1-\gamma/\sqrt{u}-\gamma\,a_u(\gamma)} +\frac{1}{\sqrt{u}}~\frac{a_u(\gamma)}{(1-\gamma/\sqrt{u})^2} \nonumber \\[3.5mm]
&  \hspace*{-1cm}=~{-}p_u(\gamma)\,\frac{\gamma+1/\sqrt{u}}{\gamma(1-\gamma/\sqrt{u})(1-p_u(\gamma))}+ \frac{p_u(\gamma)}{\sqrt{u}}~\frac{1}{\gamma(1-\gamma/\sqrt{u})}.
\end{align}
Hence, at a $\gamma$ where (\ref{EC.57}) holds, we have
\beq \label{EC.59}
\Bigl(\frac{a_u(\gamma)}{1-\gamma/\sqrt{u}}\Bigr)'=\frac{-p}{\gamma(1-p)}~\frac{\gamma+1/\sqrt{u}}{1-\gamma/\sqrt{u}} +\frac{p}{\gamma\sqrt{u}}~\frac{1}{1-\gamma/\sqrt{u}},~~~~~~u=s,t.
\eq
Therefore, at such a $\gamma$ we have
\begin{eqnarray} \label{EC.60}
& \mbox{} & \Bigl(\frac{a_s(\gamma)}{1-\gamma/\sqrt{s}}\Bigr)'-\Bigl(\frac{a_t(\gamma)}{1-\gamma/\sqrt{t}}\Bigr)' \nonumber \\[3.5mm]
& & =~\frac{-p}{\gamma(1-p)}\,\Bigl(\frac{\gamma+1/\sqrt{s}}{1-\gamma/\sqrt{s}}-\frac{\gamma+1/\sqrt{t}} {1-\gamma/\sqrt{t}}\Bigr)+\frac{p}{\gamma}\,\Bigl(\frac{1}{\sqrt{s}-\gamma}-\frac{1}{\sqrt{t}-\gamma}\Bigr) \nonumber \\[3.5mm]
& & =~\frac{p(1/\sqrt{s}-1/\sqrt{t})}{\gamma(1-\gamma/\sqrt{s})(1-\gamma/\sqrt{t})}\,\Bigl({-}\,\frac{1+\gamma^2}
{1-p}+1\Bigr)>0
\end{eqnarray}
since $s>t$, $p\in(0,1)$ and $\gamma\in(0,\sqrt{t})$. Letting $\gamma_0={\rm inf}\,\{0<\gamma<\sqrt{t}\:|\:\mbox{(\ref{EC.57}) holds}\}\in(0,\sqrt{t})$, we arrive at a contradiction in the same way as in the proof of (\ref{EC.41}). Hence, (\ref{EC.56}) holds for all $\gamma\in(0,\sqrt{t})$. Then from decreasingness of $\gamma\,a_s(\gamma)/(1-\gamma/\sqrt{s})$ as a function of $\gamma$, it is then seen that $\hat{\gamma}_s>\hat{\gamma}_t$.

The limit value of $\hat{\gamma}_s$ as $s\pr\infty$ follows on considering the equation $\gamma\,a_{\infty}(\gamma)=1/2$ in which $a_{\infty}(\gamma)$ is the unique solution of the Cohen equation in (\ref{6.4}). We still need a lemma.
\begin{lem}\label{lemfinal}$\gamma\,a_{\infty}(\gamma)$ strictly decreases in $0<\gamma<\infty$ from $1$ to $0$.
\end{lem}
\noindent{\it Proof.}~~Using
\beq \label{EC.61}
a_{\infty}'(\gamma)=\frac{-\gamma\,a_{\infty}(\gamma)}{1-\gamma\,a_{\infty}(\gamma)},
\eq
compare (\ref{EC.36}), it follows for $\gamma>0$ that
\beq \label{EC.62}
(\gamma\,a_{\infty}(\gamma))'<0\Leftrightarrow\gamma\,a_{\infty}(\gamma)>1-\gamma^2.
\eq
From \cite[(20)]{ajl}, it is seen that $\gamma\,a_{\infty}(\gamma)>1-\gamma^2$ holds for small positive $\gamma$. At a $\gamma>0$ where $\gamma\,a_{\infty}(\gamma)=1-\gamma^2$, we compute $(\gamma\,a_{\infty}(\gamma))'=0>{-}2\gamma=(1-\gamma^2)'$ and so, by the method of the proofs of (\ref{EC.41}) and (\ref{EC.56}), such a $\gamma$ does not exist. Hence, $\gamma\,a_{\infty}(\gamma)$ is strictly decreasing in $\gamma>0$. Finally, $\gamma\,a_{\infty}(\gamma)\pr1$ as $\gamma\downarrow0$ follows from \cite{ajl}, Theorem~\ref{thmm1}4, and $\gamma\,a_{\infty}(\gamma)\pr0$ as $\gamma\pr\infty$ follows from
\beq \label{EC.63}
\gamma\,a_{\infty}(\gamma)=\gamma\,f_{\infty}(\gamma-a_{\infty}(\gamma))<\gamma\,f_{\infty}(\gamma-a_{\infty}(1))\pr0 ,~~~~~~\gamma\pr\infty.
\eq
This completes the proof of the lemma. \\
\mbox{}

We conclude from Lemma~\ref{lemfinal} that $\gamma\,a_{\infty}(\gamma)=1/2$ has a unique solution $\hat{\gamma}_{\infty}$, which can be determined numerically by a two-stage Newton method, using conveniently
\beq \label{EC.64}
f_{\infty}'(\delta)={-}f_{\infty}(\delta)(\delta+f_{\infty}(\delta))~~\mbox{and~~(\ref{EC.61})}
\eq
for solving $a=f_{\infty}(\gamma-a)$ for $a=a_{\infty}(\gamma)$ and $\gamma\,a_{\infty}(\gamma)=1/2$ for $\gamma$, respectively.

The value of $\hat{\gamma}_s$ for $s=1$ follows from the Note before Theorem~\ref{thmEC21}.
\end{proof}
\end{document}